\documentclass[a4paper,11pt]{article}
\usepackage{amsmath, amsthm, amsfonts, amssymb, bbm}
\usepackage{graphicx, psfrag, color}
\usepackage{cite}
\usepackage{hyperref}
\usepackage[margin=0.9in]{geometry}

\newcommand{\ics}{\frac{1}{(2\pi\I)^2}}
\newcommand{\Lpf}{L^{\rm pf}_{N,N-n}}
\newcommand{\iid}{i.i.d.\,}

\newcommand{\Or}{\mathcal{O}}

\newcommand{\Pb}{\mathbb{P}}
\newcommand{\E}{\mathbbm{E}}
\newcommand{\Id}{\mathbbm{1}}
\newcommand{\e}{\varepsilon}
\newcommand{\I}{{\rm i}}

\newcommand{\R}{\mathbb{R}}
\newcommand{\N}{\mathbb{N}}
\newcommand{\Z}{\mathbb{Z}}
\renewcommand{\Re}{\mathrm{Re}\,}

\DeclareMathOperator{\sgn}{sgn}
\DeclareMathOperator*{\pf}{\mathrm{pf}}
\newcommand{\bra}[1]{\left\langle #1 \right|}
\newcommand{\ket}[1]{\left| #1 \right\rangle}
\newcommand{\braket}[2]{\left\langle #1 \left| #2 \right\rangle \right.}
\newcommand{\ketbra}[2]{\left| #1 \right\rangle \left\langle #2 \right|}

\newcommand{\zcd}{\raisebox{-0.5\height}{\includegraphics[scale=0.04]{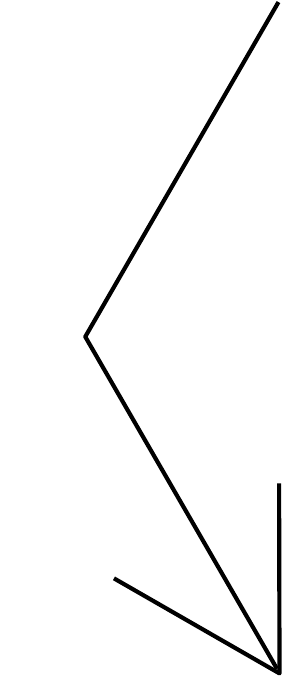}}}
\newcommand{\wcu}{\raisebox{-0.5\height}{\includegraphics[scale=0.04]{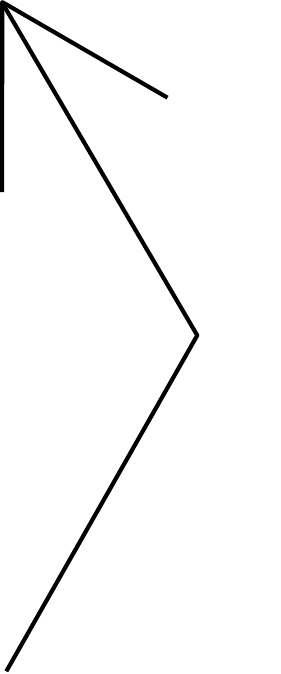}}}

\DeclareMathAlphabet{\mathpzc}{OT1}{pzc}{m}{it}

\newtheorem{prop}{Proposition}[section]
\newtheorem{thm}[prop]{Theorem}
\newtheorem{lem}[prop]{Lemma}

\newtheorem{cor}[prop]{Corollary}

\newtheorem{cla}[prop]{Claim}

\newtheorem{remark}[prop]{Remark}

\newenvironment{rem}{\begin{remark}\normalfont}{\end{remark}}

\numberwithin{equation}{section}
\allowdisplaybreaks

\title{Stationary half-space last passage percolation}
\author{D. Betea\thanks{Institute for Applied Mathematics, Bonn University, Endenicher Allee 60, 53115 Bonn, Germany. E-mail: {\tt dan.betea@gmail.com}}
\and P. L. Ferrari\thanks{Institute for Applied Mathematics, Bonn University, Endenicher Allee 60, 53115 Bonn, Germany. E-mail: {\tt ferrari@uni-bonn.de}}
\and A. Occelli\thanks{Institute for Applied Mathematics, Bonn University, Endenicher Allee 60, 53115 Bonn, Germany. E-mail: {\tt occelli@iam.uni-bonn.de}}}
\date{January 21, 2020}

\begin{document}
\sloppy

\maketitle

\begin{abstract}
In this paper we study stationary last passage percolation (LPP) with exponential weights and in half-space geometry. We determine the limiting distribution of the last passage time in a critical window close to the origin. The result is a new two-parameter family of distributions: one parameter for the strength of the diagonal bounding the half-space (strength of the source at the origin in the equivalent TASEP language) and the other for the distance of the point of observation from the origin. It should be compared with the one-parameter family giving the Baik--Rains distributions for full-space geometry. We finally show that far enough away from the characteristic line, our distributions indeed converge to the Baik--Rains family. We derive our results using a related inhomogeneous integrable model having Pfaffian correlations, together with careful analytic continuation, and steepest descent analysis.
\end{abstract}

\section{Introduction}

\paragraph{Background and motivation.} A stochastic growth model in the one-dimensional Kardar--Parisi--Zhang (KPZ) universality class~\cite{KPZ86} describes the evolution of a height function $h(x, t)$ at position $x$ and time $t$ subject to a stochastic and local microscopic evolution. On a macroscopic scale, that is with space of order $t$, the evolution of the height function evolves according to a certain PDE and one has a non-random limit shape.

The following, among others, belong to the the KPZ class: the KPZ equation, directed random polymer models (where the free energy plays the role of the height function), their zero-temperature limits known as last passage percolation, and interacting particle systems like the exclusion process. Some of these models have been analyzed in the last two decades for many classes of initial conditions and/or boundary conditions. The fluctuations of the height function are of order $t^{1/3}$ and the correlation length scales as $t^{2/3}$, as conjectured in~\cite{FNS77,BKS85}\footnote{This holds true around points with smooth limit shape. Around shocks there are some differences, see e.g.~\cite{FN13,FN16,FF94b,FGN17,Nej17}.}.

In particular, it is known that the limiting processes depend on subclasses of initial conditions. In full-space, that is for $x\in\R$ or $x\in\Z$, one sees the Airy$_2$ process around curved limit shape points~\cite{PS02,Jo03b,BF07}, with one-point distribution the GUE Tracy--Widom distribution~\cite{TW94}, discovered in~\cite{BDJ99} and shown later to hold for a variety of models in the KPZ class~\cite{Jo00b,SS10,ACQ10,SS10b,FV13,BCF12,Bar14}. For flat limit shapes and non-random initial conditions, the limit process is known as the Airy$_1$ process~\cite{Sas05,BFPS06} with the GOE Tracy--Widom as one-point distribution~\cite{TW96}, see~\cite{BR99,BR99b,PS00}. Finally, stationary initial conditions also lead to flat limit shapes and the Airy$_{\rm stat}$ process~\cite{BFP09}, having the Baik--Rains distribution as the one-point distribution~\cite{BR00,FS05a,BCFV14,IS12,Agg16,SI17,SI17b,IMS19}. The stationary model is obtained as a limit of some specific two-sided random initial condition\footnote{The random initial condition on the two sides is recovered using boundary sources, by the use of some Burke-type property~\cite{Bur56,DMO05}, as shown for the exclusion process in~\cite{PS02b}.}. Further results for random but not necessarily stationary initial conditions are also known~\cite{CLW16,FO17,CFS16,QR16}.

For further details and recent developments around the KPZ universality class, see also the following surveys and lecture notes:~\cite{FS10,Cor11,QS15,BG12,Qua11,Fer10b,Tak16,Zyg18}.

In this paper we consider a stationary model in half-space, where the latter means having a height function $h(x,t)$ defined only on $x\in\N$ (or $x\in\R_+$). Our model, called stationary half-space last passage percolation (LPP), is defined in Section~\ref{sec:LPP}. In this geometry there are considerably fewer results compared to the case of full-space geometry. Of course, one has to prescribe the dynamics at site $x=0$. If the influence on the height function of the growth mechanism at $x=0$ is very strong, then close to the origin one will essentially see fluctuations induced by it, and since the dynamics in KPZ models has to be local (in space but also in time), one will observe Gaussian fluctuations. If the influence of the origin is small, then it will not be seen in the asymptotic behavior. Between the two situations there is typically a critical value where a third different distribution function is observed. Furthermore, under a critical scaling, one obtains a family of distributions interpolating between the two extremes. For some versions of half-space LPP and related stochastic growth models (with non-random initial conditions) this has indeed been proven: one has a transition of the one-point distribution from Gaussian to GOE Tracy--Widom at the critical value, and GSE Tracy--Widom distribution~\cite{BR99b,SI03,BBC18,KLD19}\footnote{The results are obtained through a description in terms of Pfaffian structures~\cite{BR01b, RB04, Ra00, FR07, Gho17, BBCW17, BBNV18,BZ17,BZ19}.} Furthermore, the limit process under critical scaling around the origin is also analyzed and the transition processes have been characterized~\cite{SI03,BBCS17,BBNV18,BBCS17b}.

However, the limiting distribution of the stationary LPP in half-space remained unresolved. In this paper we close this gap: in Theorem~\ref{thm:main_thm_fin} we determine the distribution function of the stationary LPP for the finite size system and in Theorem~\ref{thm:main_thm_asymp} we determine the large time limiting distribution under critical scaling.

The physical difference with the full-space problem is the source at the origin, and this makes the half-space problem richer. The influence of the boundary process is still visible in the limiting distribution: we have a new two-parameter family of distribution functions, one parameter for the strength of the source at the origin, the second for the distance from the diagonal. As a comparison, in the full-space analogue it was a one-parameter family of distributions with the parameter describing the distance from the characteristic line.

The limiting distribution we observe should be universal within the KPZ universality class to the same extent that the Baik--Rains distribution is for the case of full-space. By choosing the observation position far from the diagonal and setting the strength of the source to stay around a characteristic line from the origin, we furthermore recover, in the limit, the aforementioned Baik--Rains distribution from full-space. See Theorem~\ref{thm:LimitToBR}.

A second reason for the study of the stationary case is the following. In the full-space case, it was shown in~\cite{FO18} that the first order correction of the time-time covariance for times close to each others on a macroscopic scale is governed by the variance of the Baik--Rains distribution. The reason is that the system locally converges to equilibrium. Thus, for any comparable study in half-space geometry, the knowledge of the stationary limiting distribution and/or process is necessary.

Concerning the methods used in this paper, there are some similarities but also important differences with respect those from the case of full-space geometry studied in~\cite{FS05a,BFP09}. To identify the stationary model, it has been useful to start from the exclusion process analogy, for which the stationary measures in half-space were obtained in~\cite{Lig75,Lig77}. Next we observe that, as in full-space, the desired distribution function can be obtained in a two-step procedure. First we study a two-parameter integrable model which has a Pfaffian structure. By a so-called shift argument, we can write the two-parameter distribution function in terms of the distribution function of the Pfaffian model. Finally we need to perform a limit when the sum of the two parameters goes to zero. This is achieved through analytic continuation. This last step turned out to be considerably more complicated that in the full-space case~\cite{FS05a} as some exact cancellations of diverging terms happened only using the $2\times 2$ structure of the Pfaffian kernel. Such issues did not show up in the full-space analysis.

\paragraph{Outline.} In Section~\ref{sec:Results} we define last passage percolation and its half-space stationary version (Sections~\ref{sec:LPP} and~\ref{sec:stat_res}), and state the main results of this paper: the finite size result in Sections~\ref{sec:fin_time_res}, the asymptotic result in \ref{sec:ass_res}, and the limit transition to Baik--Rains in Section~\ref{sec:limitTransitionBR}. The finite-time formula for the stationary LPP in half-space, Theorem~\ref{thm:main_thm_fin}, is derived in Section~\ref{sec:proofFiniteN}: we start from a half-space integrable model (Section~\ref{sec:int}) with known Pfaffian correlations; a shift argument and further finite-rank kernel decomposition give amenable formulas for taking the limit to the stationary case (Section~\ref{sec:int_to_stat}); we take the latter limit using a careful analytic continuation argument (Section~\ref{sec:analytic_continuation}) and this yields the finite result. In Section~\ref{sec:proofAsymptotic} we prove the asymptotic result of Theorem~\ref{thm:main_thm_asymp}. The limit to Baik--Rains, namely Theorem~\ref{thm:LimitToBR}, is proven in Section~\ref{sec:limitToBR}. We list, in the Appendices, some background material on Pfaffian processes and some otherwise useful auxiliary results, including the limit from geometric to exponential model.

\paragraph{Notations.} Throughout this work we handle numerous complex integrals. To simplify matters, we choose the following special notation for types of contours we will often encounter. First, $\Gamma_{I}$ will indicate any simple counter-clockwise contour around the set of points $I$. We remark that sometimes such a contour will just be a disjoint union of simple counter-clockwise contours each encircling one of the points in $I$. In the large time asymptotics sections we use the following notation for the typical Airy contours, denoting with ${}_I\raisebox{-0.5\height}{\includegraphics[scale=0.06]{z_cont_down}}\,{}_J$ a down-oriented contour coming in a straight line from $\exp(\pi i /3)\infty$ to a point on the real line to the right of $I$ and to the left of $J$, and continuing in a straight line to $\exp(5 \pi i /3)\infty$, and with ${}_{I}\,\raisebox{-0.5\height}{\includegraphics[scale=0.06]{w_cont_up}} {}_{J}$ an up-oriented contour from $\exp(4 \pi i /3)\infty$ to $\exp(2\pi i/3)\infty$. Examples are depicted in Figure~\ref{fig:airy_contours}.

For two functions $f, g$ we use (the usual) bra-ket notation as follows: the scalar product on $L^2(s,\infty)$ is denoted by
\begin{equation}
 \braket{f}{g} = \int_s^\infty f(x) g(x) dx
\end{equation}
while by $\ketbra{f}{g}$ we denote the outer product kernel
\begin{equation}
 \ketbra{f}{g} (x, y) = f(x) g(y).
\end{equation}
We extend this notation to arguments of bras and kets that might be row or column vectors. For instance
\begin{equation}
\braket{f_1 \quad f_2} { \begin{pmatrix} a_{11} & a_{12} \\ a_{21} & a_{22} \end{pmatrix} \begin{pmatrix} g_1 \\ g_2 \end{pmatrix} } = \braket{f_1}{a_{11} g_1 + a_{12} g_2} + \braket{f_2}{a_{21} g_1 + a_{22} g_2}
\end{equation}
is a sum of four scalar products (where the $a_{ij}:=a_{ij}(x,y)$'s are integral operators), while
\begin{equation}
\ketbra{\begin{array}{c} f_1 \\ f_2 \end{array} } {g_1 \quad g_2} (x, y) = \begin{pmatrix} f_1(x) g_1(y) & f_1(x) g_2(y) \\ f_2(x) g_1(y) & f_2(x) g_2(y) \end{pmatrix}
\end{equation}
is a $2 \times 2$ matrix kernel.

In this paper we will define many kernels and functions that depend on the parameters $\alpha$ and $\beta$, and will consider their limits for $\beta\to-\alpha$. In order to distinguish these cases, we will denote in up-right sans-serif $\mathsf{font}$, the limits as $\beta \to - \alpha$ of the various objects, which will depend explicitly on $\beta$. The ones that are independent of $\beta$ are left with their font unchanged. With this choice, for example, we'll use
 \begin{equation}
 {\mathsf{K}} = \lim_{\beta \to -\alpha} {K}
 \end{equation}
 where implicitly $K$ depends on parameters $\alpha$ and $\beta$.

\paragraph{Acknowledgements.} This work is supported by the German Research Foundation through the Collaborative Research Center 1060 ``The Mathematics of Emergent Effects'', project B04, and by the Deutsche Forschungs-gemeinschaft (DFG, German Research Foundation) under Germany's Excellence Strategy - GZ 2047/1, Projekt ID 390685813.

The authors would furthermore like to thank G. Barraquand, E. Bisi, J. Bouttier, P. Le Doussal, E. Emrah, S. Grosskinsky, P. Nejjar, T. Sasamoto, H. Spohn, B. Vet\H{o}, and N. Zygouras for useful comments and suggestions regarding this manuscript and/or their related papers. We are also very grateful to an anonymous referee for careful reading of our manuscript and a number of constructive remarks.

\section{Model and main results} \label{sec:Results}

\subsection{Last passage percolation} \label{sec:LPP}

Before going to the specific model studied in this paper, let us recall the more generic last passage percolation (LPP) model on $\Z^2$ and explain where the denomination half-space comes from.

Consider independent random variables $\{\omega_{i,j},\,i,j\in\Z\}$. An \emph{up-right path} $\pi$ on $\Z^2$ from a point $A$ to a point $E$ is a sequence of points $(\pi(0),\pi(1),\ldots,\pi(n))$ in $\Z^2$ such that \mbox{$\pi(k+1)-\pi(k)\in \{(0,1),(1,0)\}$}, with $\pi(0)=A$ and $\pi(n)=E$, and where $n$ is called the length $\ell(\pi)$ of $\pi$. Now, given a set of points $S_A$ and $E$, one defines the \emph{last passage time} $L_{S_A\to E}$ as
\begin{equation}\label{eq:3.2}
L_{S_A\to E}=\max_{\begin{subarray}{c}\pi:A\to E\\A\in S_A\end{subarray}} \sum_{1\leq k\leq \ell(\pi)} \omega_{\pi(k)}.
\end{equation}
Finally, we denote by $\pi^{\rm max}_{S_A\to E}$ any maximizer of the last passage time $L_{S_A\to E}$. For continuous random variables, the maximizer is a.s.\,unique. In this paper we consider exponentially distributed random variables, which give the well-known connection with the totally asymmetric simple exclusion process (TASEP).

TASEP is an interacting particle system on $\Z$ with state space $\Omega=\{0,1\}^\Z$. For a configuration $\eta\in\Omega$, $\eta=(\eta_j,j\in\Z)$, $\eta_j$ is the occupation variable at site $j$, which is $1$ if and only if $j$ is occupied by a particle. TASEP has generator $\cal L$ given by~\cite{Li99}
\begin{equation}\label{eq:1.1}
{\cal L}f(\eta)=\sum_{j\in\Z}\eta_j(1-\eta_{j+1})\big(f(\eta^{j,j+1})-f(\eta)\big)
\end{equation}
where $f$'s are local functions (depending only on finitely many sites) and $\eta^{j,j+1}$ denotes the configuration $\eta$ with occupations at sites $j$ and \mbox{$j+1$} interchanged. Notice that for TASEP the ordering of particles is preserved. That is, if initially one orders particles from right to left as
\begin{equation}
  \ldots < x_2(0) < x_1(0) < 0 \leq x_0(0)< x_{-1}(0)< \ldots
\end{equation}
then for all times $t\geq 0$ also $x_{n+1}(t)<x_n(t)$, $n\in\Z$.

The connection between TASEP and LPP is as follows. Take $\omega_{i,j}$ to be the waiting time of particle $j$ to jump from site $i-j-1$ to site $i-j$. Then $\omega_{i,j}$ are exponential random variables. Further, choosing the set $S_A=\{(u,k)\in\Z^2: u=k+x_k(0), k\in\Z\}$, we have that
\begin{equation}\label{eq:2.4}
\Pb\left(L_{{S_A}\to (m,n)}\leq t\right)=\Pb\left(x_n(t)\geq m-n\right)=\Pb\left(h_t(m-n)\geq m+n\right),
\end{equation}
where $h_t$ is the standard height function associated with TASEP.

The denomination full-space (respectively half-space) LPP comes from the fact that the height function and particles live on $\Z$ (respectively $\N$). The relation~\eqref{eq:2.4} means that for half-space, i.e.\ $x=m-n\geq 0$, the random variables in LPP are restricted to $\{(m,n)| m\geq n\}$ (equivalently, we can think that the other random variables are set to be $0$).

In the framework of some interacting particle systems, with TASEP being the simplest case, Liggett studied the invariant measures for the full-space geometry~\cite{Lig75,Lig77}. To achieve his result, he first considered a finite system from which the half-space model is a simple limiting case. In particular, for TASEP defined on $\N$ with particles entering at the origin with a given rate $\lambda\in[0,1]$, i.e.\,the origin playing the role of a reservoir of particles, he showed that the stationary measure with particle density $\rho=\lambda$ on $\N$ is a product measure. For this reason, the LPP analogue is obtained by considering weights on the diagonal as being exponentially distributed of parameter $\rho$ (below we set $\rho=\tfrac12+\alpha$), while the random initial condition in $\N$ can be replaced by Burke's theorem~\cite{Bur56} with a first row in the LPP geometry having exponentially distributed random variables of parameter $1-\rho$.

It is worth mentioning that for half-space TASEP the stationary measures are not unique. In particular, there are also not product measures, see Theorem~1.8 of~\cite{Lig75}. A representation using matrix product ansatz is given in~\cite[Theorem 3.2]{Gro04}. The mapping between LPP and TASEP would imply that the $\omega_{ij}$ of the corresponding LPP are not independent random variables anymore. Our techniques do not apply however in such cases.

\subsection{The stationary half-space model} \label{sec:stat_res}

Let us now focus on the half-space LPP model. On the set ${\cal D}=\{(i,j)\in\Z^2 | 1\leq j \leq i\}$ we consider independent non-negative random variables $\{\omega_{i,j}\}_{(i,j)\in {\cal D}}$. Then, the half-space LPP time to the point $(n,m)$ (for $m\leq n$), denoted $L_{n,m}$, is given by
\begin{equation}
L_{n,m}=\max_{\pi:(1,1)\to (n,m)} \sum_{(i,j)\in\pi} \omega_{i,j}
\end{equation}
where the maximum is over up-right paths in ${\cal D}$ from $(1, 1)$ to $(m,n)$, i.e.\,paths with increments in $\{(0,1),(1,0)\}$.

We are interested in the stationary version of this model which, as we will see, can be obtained as follows. Let us write $\mathrm{Exp}(a)$ for an exponential random variable with parameter $a>0$. Then, the stationary version is given by setting
\begin{equation} \label{eq:stat_wts}
  \omega_{i, j} = \begin{cases}
    \mathrm{Exp}\left( \frac{1}{2} + \alpha \right), & i=j>1,\\
    \mathrm{Exp}\left( \frac{1}{2} - \alpha \right), & j=1, i>1, \\
    0, & \textrm{if}\ i=j=1, \\
    \mathrm{Exp}(1), &\textrm{otherwise}
  \end{cases}
\end{equation}
where $\alpha \in (-1/2, 1/2)$ is a fixed parameter. A schematic depiction is drawn in Figure~\ref{fig:stat_lpp}.

\begin{figure}[t!]
  \centering
   \includegraphics[height=5.5cm]{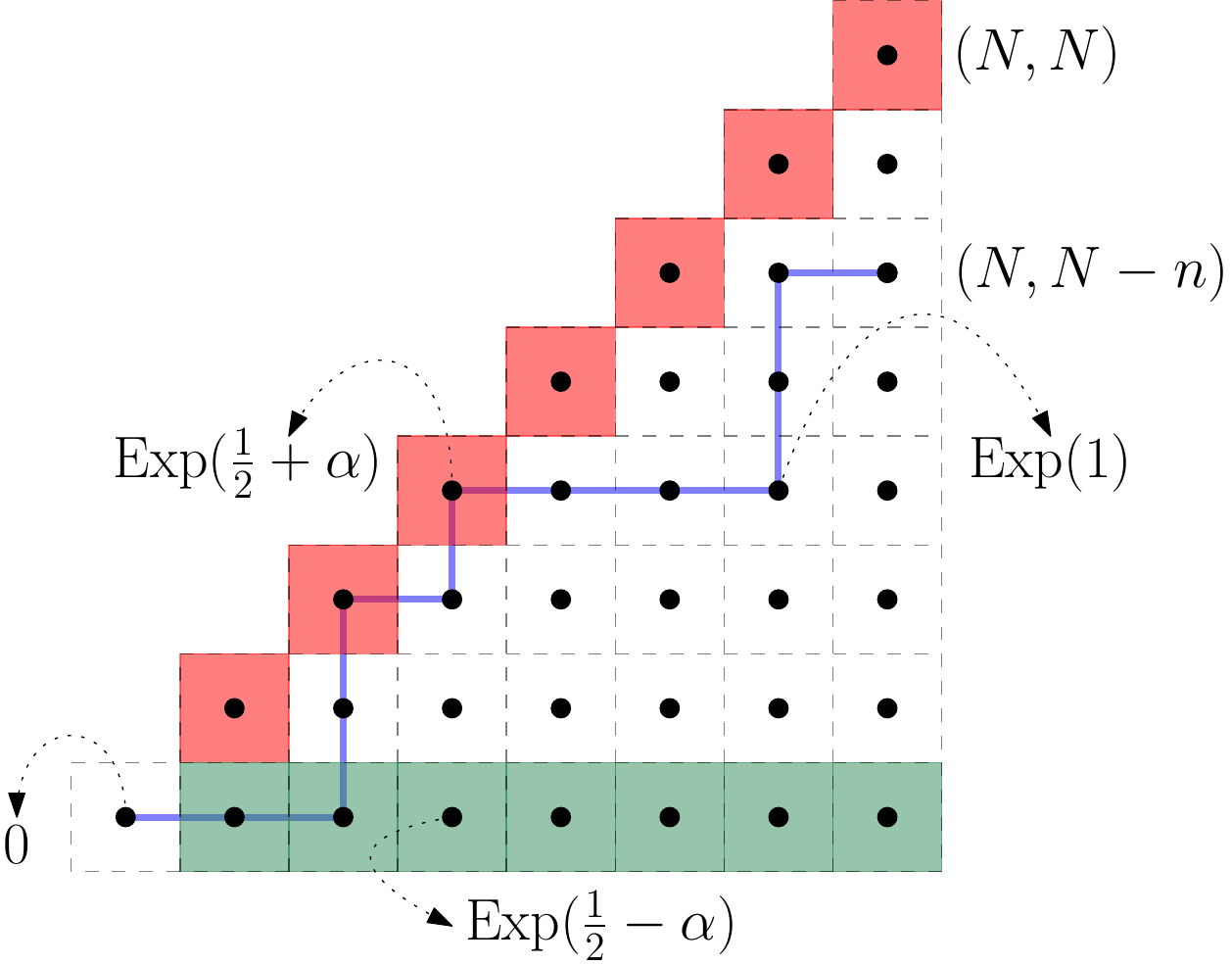}
   \caption{A possible LPP path (polymer) from $(1, 1)$ to $(N, N-n)$ for $(N, n) = (8, 2)$ in the stationary case. The dots are independent random variables: deterministically 0 at the origin, $\mathrm{Exp}(\tfrac12+\alpha)$ (respectively $\mathrm{Exp}(\tfrac12 - \alpha)$) on the rest of the diagonal (respectively the bottom line), and $\mathrm{Exp}(1)$ everywhere else in the bulk.}
   \label{fig:stat_lpp}
\end{figure}

This model is stationary in the sense of~\cite{BCS06}, i.e.\,it has stationary increments as stated in the following lemmas.

\begin{lem}(Half-space version of~\cite[Lemma 4.1]{BCS06})\label{lem:stationarity}
  Fix any $i, j\geq1$ with $i<j$. The following three random variables are jointly independent and distributed as follows: the increment along the horizontal direction
  \begin{equation}
   H_{i+1, j+1} = L_{i+1,j+1}-L_{i,j+1}
  \end{equation}
 is an $\textup{Exp} \left(\tfrac12-\alpha \right)$ random variable; the increment along the vertical direction
  \begin{equation}
   V_{i+1, j+1} = L_{i+1, j+1}-L_{i+1, j}
  \end{equation}
 is an $\textup{Exp} \left( \tfrac12+\alpha \right)$ random variable; finally, the minimum of the horizontal and vertical increments at a vertex $(i,j)$, defined by
  \begin{equation}
   X_{i, j} = \min(H_{i+1, j}, V_{i, j+1}) = \min( L_{i+1, j} - L_{i, j}, L_{i, j+1} - L_{i, j} )
  \end{equation}
 is an $\textup{Exp(1)}$ random variable.
 \end{lem}

 \begin{proof}
 The result is an immediate consequence of Lemma 4.1 of~\cite{BCS06} for the full-space model. It follows by induction: if $H_{i+1, j}$ and $V_{i, j+1}$ are independent $\textup{Exp} \left(\tfrac12-\alpha \right)$ and $\textup{Exp} \left(\tfrac12 + \alpha \right)$ respectively, then the conclusion follows. The base case is simple: $H_{2, 1} = \omega_{2,1} \sim \textup{Exp} \left(\tfrac12 - \alpha \right)$ and $V_{2,2} = \omega_{2,2} \sim \textup{Exp} \left(\tfrac12 + \alpha \right)$ and both are independent.
 \end{proof}

\begin{lem} (Half-space version of~\cite[Lemma 4.2]{BCS06})\label{lem:stationarity_2}
  Fix $\Pi$ any down-right path in half-space\footnote{A path of adjacent vertices with down steps $(-1, 0)$ or right steps $(0, 1)$.}, from the diagonal to the horizontal axis. The increments along $\Pi$ are jointly independent, the horizontal ones being $\textup{Exp} \left(\tfrac12-\alpha \right)$ and the vertical ones $\textup{Exp} \left(\tfrac12 + \alpha \right)$ random variables. Moreover, they are independent from the \iid random variables $X_{i, j} = \min( L_{i+1, j} - L_{i, j}, L_{i, j+1} - L_{i, j} )$ for $(i, j)$ any point strictly below and to the left of $\Pi$.
\end{lem}

\begin{proof}
  Lemma 4.2 of~\cite{BCS06} applies mutatis mutandis.
\end{proof}

\begin{rem}
For the stationary model, we do not have a good formula to study the statistics of $L_{N,N-n}$. However, as it was already the case for the full-space stationary LPP, there is a way to recover it in a two-step procedure. First we consider a related LPP model, with two parameters, $\alpha,\beta$, for which we have an explicit formula of the distribution in terms of a Fredholm Pfaffian. The distribution of the original model is recovered by a standard shift argument and an analytic continuation. This last step turns out to be far from trivial and different from both the one for the full-space stationary model~\cite{FS05a} and its multi-point extension~\cite{BFP09}.
\end{rem}

\subsection{Finite time distribution for the stationary LPP} \label{sec:fin_time_res}

The first result of this paper is a formula for the distribution of the stationary LPP time $L_{N, N-n}$. In order to state the result, we need to introduce a few functions and kernels which appear in the final expression. Let us define the following $f$ functions
\begin{equation}
  f_{+}^{-\alpha}(x) = \Phi(x, -\alpha) \left( \tfrac12 + \alpha \right)^n, \quad f_{-}^{-\alpha}(x) = \frac{\Phi(x,-\alpha)} {\left( \tfrac12 -\alpha \right)^n}
\end{equation}
where $\Phi(x, z) = e^{-xz} \left[ \frac{\tfrac12 + z}{\tfrac12 - z} \right]^{N-1}$, the following $g$ functions
\begin{alignat}{2}
 g_1(x)& = \oint\limits_{\Gamma_{1/2}} \frac{dz}{2\pi\I} \Phi (x,z)\left( \tfrac12-z \right)^n\frac{z+\alpha}{2z},\quad &g_2(x)& = \oint\limits_{\Gamma_{1/2,\alpha}} \frac{dz}{2\pi\I} \frac{\Phi(x,z)}{\left( \tfrac12+z \right)^n}\frac{1}{z-\alpha}, \nonumber \\
\mathsf{g}_3(x) &= \oint\limits_{\Gamma_{1/2}} \frac{dz}{2\pi\I} \Phi(x, z) \left( \tfrac12-z \right)^n  \frac{1}{z+\alpha},\quad &\mathsf{g}_4(x) &= \!\!\oint\limits_{\Gamma_{1/2, \pm \alpha}} \!\!\frac{dz}{2\pi\I} \frac{\Phi(x, z)}{\left( \tfrac12+z \right)^n} \frac{2z}{(z - \alpha) (z + \alpha)^2}
\end{alignat}
and finally let
\begin{equation}
e^{\alpha} (s) = - \oint\limits_{\Gamma_{1/2,\alpha}}\frac{dz}{2\pi\I} \frac{\Phi(s,z)}{\Phi(s,\alpha)}\frac{(\tfrac12+\alpha)^n}{(\tfrac12+z)^n}\frac{1}{(z-\alpha)^2}.
\end{equation}
Also, define the following \emph{anti-symmetric} kernel $\overline{\mathsf{K}}$:
 \begin{equation}
  \begin{split}
   \overline{\mathsf{K}}_{11}(x,y) =& - \!\! \oint\limits_{\Gamma_{1/2}} \frac{dz}{2\pi\I} \!\! \oint\limits_{\Gamma_{-1/2}} \!\! \frac{dw}{2\pi\I}\frac{\Phi(x,z)}{\Phi(y,w)}\left[(\tfrac12-z)(\tfrac12+w)\right]^n\frac{(z-\alpha)(w+\alpha)(z+w)}{4zw(z-w)}, \\
   \overline{\mathsf{K}}_{12}(x,y) =& -\!\!\! \oint\limits_{\Gamma_{1/2}} \frac{dz}{2\pi\I} \!\!\! \oint\limits_{\Gamma_{-1/2,\alpha}} \!\!\! \frac{dw}{2\pi\I}\frac{\Phi(x,z)}{\Phi(y,w)}\left[\frac{\tfrac12-z}{\tfrac12-w}\right]^n\frac{z-\alpha}{w-\alpha}\frac{z+w}{2z(z-w)} \\
               =& -\overline{\mathsf{K}}_{21} (y,x),\\
   \overline{\mathsf{K}}_{22}(x,y) =&\ \varepsilon(x,y)+\oint \frac{dz}{2\pi\I}\oint\frac{dw}{2\pi\I} \frac{\Phi(x,z)}{\Phi(y,w)} \frac{1}{\left[(\tfrac12+z)(\tfrac12-w)\right]^n}\frac{1}{z-w}\left(\frac{1}{z+\alpha}+\frac{1}{w-\alpha}\right)
  \end{split}
 \end{equation}
  where the integration contours for $\overline{\mathsf{K}}_{22}$ are $\Gamma_{1/2,-\alpha}\times\Gamma_{-1/2}$ for the term with $1/(z+\alpha)$ and $\Gamma_{1/2}\times\Gamma_{-1/2,\alpha}$ for the term with $1/(w-\alpha)$, and where $\varepsilon = \varepsilon_0+\varepsilon_1$ with
\begin{equation}
\varepsilon_0(x,y) =  - \sgn(x-y) \frac{ e^{\alpha |x-y|} } {\left( \frac{1}{4} - \alpha^2 \right)^n}, \quad  \varepsilon_1(x, y) =  - \sgn(x-y)\oint\limits_{\Gamma_{1/2}} \frac{dz}{2\pi\I} \frac{2z e^{-z |x-y|}}{(z^2-\alpha^2) \left( \frac{1}{4} - z^2 \right)^n}.
\end{equation}
Finally, define
\begin{equation}\label{eq2.16}
 \begin{aligned}
 \widetilde{\mathsf{K}}_{12}(x,y) &=- \!\! \oint\limits_{\Gamma_{1/2}} \!\! \frac{dz}{2\pi\I} \!\! \oint\limits_{\Gamma_{-1/2}} \!\! \frac{dw}{2\pi\I}\frac{\Phi(x,z)}{\Phi(y,w)} \left[\frac{\tfrac12-z}{\tfrac12-w}\right]^n\frac{z-\alpha}{w-\alpha}  \frac{z+w}{2z(z-w)},\\
 \widetilde{\mathsf{K}}_{22}(x,y) &= \!\!\! \oint\limits_{\Gamma_{1/2, - \alpha}} \!\!\! \frac{dz}{2\pi\I} \!\!\! \oint\limits_{\Gamma_{-1/2}} \!\!\! \frac{dw}{2\pi\I} \frac{\Phi(x,z)}{\Phi(y,w)} \frac{1}{\left[(\tfrac12+z)(\tfrac12-w)\right]^n}\frac{1}{(z+\alpha)(w-\alpha)} \frac{z+w}{z-w}
 \end{aligned}
 \end{equation}
and
  \begin{equation}
  \begin{split}
  h_1 &= \widetilde{\mathsf{K}}_{22} f_{+}^{-\alpha} + \varepsilon_1 f_{+}^{-\alpha} - \mathsf{g}_4 - j^\alpha(s, \cdot), \\
  h_2 &= \widetilde{\mathsf{K}}_{12} f_{+}^{-\alpha} + \mathsf{g}_3
  \end{split}
 \end{equation}
 with
 \begin{equation}\label{eq:214}
  j^{\alpha}(s, y) =\left( \frac{\sinh \alpha (y-s)}{\alpha} + (y-s)e^{\alpha(y-s)}\right)f_{-}^{-\alpha} (s).
 \end{equation}
In \eqref{eq2.16} the product between kernel and functions are on $L^2(s,\infty)$.

Our first main theorem, a finite size result, is as follows.

\begin{thm} \label{thm:main_thm_fin}
 Let $\alpha \in (-1/2, 1/2)$ be a real number and $1 \leq N$, $0 \leq n \leq N-1$ be positive integers. Let $L_{N,N-n}$ be the stationary LPP time from $(1, 1)$ to $(N, N-n)$ in the model of weights given by~\eqref{eq:stat_wts}. Then
 \begin{equation}\label{eq:final_dist}
 \begin{split}
  \Pb (L_{N,N-n} \leq s) = \partial_s   \Bigg\{ \pf(J - \overline{\mathsf{K}}) \cdot \Bigg[ e^{\alpha}(s) - \braket{-g_1 \quad g_2} {(\Id-J^{-1}\overline{\mathsf{K}})^{-1} \begin{pmatrix} -h_1 \\ h_2 \end{pmatrix} } \Bigg] \Bigg\}
  \end{split}
 \end{equation}
 where the Fredholm Pfaffian is taken over $L^2(s,\infty) \times L^2(s,\infty)$ and where $J = \left( \begin{smallmatrix} 0 & 1 \\ -1 & 0 \end{smallmatrix} \right)$.
\end{thm}

The proof of Theorem~\ref{thm:main_thm_fin} is carried out in Section~\ref{sec:proofFiniteN}.

\begin{rem} \label{rem:k_bar}
 We remark that our kernel $\overline{\mathsf{K}}$ with parameter $\alpha$ and size $N$ is identical to the integrable kernel $K$ of Section~\ref{sec:int} with parameters $-\alpha$ and $\beta = 1/2$ and size $N-1$. The latter corresponds to a model that has been studied in~\cite{BBCS17}. As a consequence, the Pfaffian $\pf(J - \overline{\mathsf{K}})$ on $L^2(s,\infty) \times L^2(s,\infty)$ is the distribution of the corresponding LPP and is in $(0, 1)$ for any fixed $s$.
\end{rem}

\subsection{Asymptotic result} \label{sec:ass_res}

In order to discuss the scaling limit, we first have to determine: (a) the limit shape approximation and (b) the position of the end-point $(N,N-n)$ which is connected with the origin by the characteristic line. The reason is that if one does not take a point in a $N^{2/3}$ neighborhood of the characteristic leaving from the origin, then one will not see the $N^{1/3}$ fluctuations typical for KPZ models. Rather one will only see Gaussian fluctuations whose origin is in the boundary terms, namely the polymer spends a time much larger that $N^{2/3}$ either on the diagonal or in the first row.

Concerning the limit shape, notice that $L_{N,N-\eta N}= L_{N,1}+\left(L_{N,N-\eta N}-L_{N,1}\right)$. The two terms are not independent, but individually are sums of independent random variables, see Lemmas~\ref{lem:stationarity} and~\ref{lem:stationarity_2}. Thus one expects that for $N\gg 1$,
\begin{equation}\label{eq:216}
L_{N,N-\eta N}\simeq \frac{N}{\tfrac12-\alpha}+\frac{(1-\eta)N}{\tfrac12+\alpha} = \frac{N}{\tfrac14-\alpha^2}-\frac{\eta N}{\tfrac12+\alpha}.
\end{equation}

For a stationary situation in TASEP with particle density $\rho$, the characteristic line has speed $1-2\rho$. In terms of last passage percolation, the density $\rho$ becomes a parameter $\alpha=\rho-1/2$ (see e.g.~\cite{PS01}) and the speed $1-2\rho$ becomes a slope in the LPP geometry given by $y/x=\rho^2/(1-\rho)^2$ (see e.g.~\cite{BFP09}). Thus, with $x=N$ and $y=N-\eta N$ we have
\begin{equation}\label{eq:217}
n=\eta N=-\frac{2\alpha}{(\tfrac12-\alpha)^2}N.
\end{equation}
Therefore, to obtain a non-trivial scaling limit, given the value of $\alpha$, one needs to choose $n$ in a $\Or(N^{2/3})$-neighborhood of~\eqref{eq:217} and to consider fluctuations on a $N^{1/3}$-scale around~\eqref{eq:216}. Choosing an end-point order $N$ away from~\eqref{eq:217} leads to Gaussian behavior on the $N^{1/2}$ scale, for the maximizer will spend $\Or(N)$ of its time either in the first row or on the diagonal in that case. Recalling that $n\geq 0$, \eqref{eq:217} cannot hold for $\alpha>0$. In that case, the polymer will spend a macroscopic portion of its time on the diagonal and will have Gaussian fluctuations.

In this paper we consider the critical scaling where $\alpha$ is close to $0$, namely we set
\begin{equation}
\alpha =  \delta 2^{-4/3} N^{-1/3},\quad \eta N=n=u 2^{5/3} N^{2/3}.
\end{equation}
With this choice we have
\begin{equation}
L_{N,N-\eta N}\simeq 4 N -2 u 2^{5/3} N^{2/3} +\delta(2u+\delta) 2^{4/3} N^{1/3}
\end{equation}
and
\begin{equation}\label{eq:220}
-\frac{2\alpha}{(\tfrac12-\alpha)^2}N = -\delta 2^{5/3}N^{2/3}+\Or(N^{1/3}).
\end{equation}

\begin{rem} \label{rem:simple_scaling}
We decided not to include the $\Or(N^{1/3})$ term $\delta(2u+\delta) 2^{4/3} N^{1/3}$ of the limit shape approximation in our calculations below, since many formulas are more compact without it. That is, we consider the scaling
\begin{equation}\label{eq:ScalingS}
s=4 N -2 u 2^{5/3}N^{2/3} + S\, 2^{4/3}N^{1/3}.
\end{equation}
However, it has to be taken into account if one wants to determine various limits, e.g.\,$u\to\infty$ and/or $\delta\to\pm\infty$. In these limits, we first have to replace $S$ by $S+\delta(2u+\delta)$. Also, when taking $u\to\infty$, by \eqref{eq:220}, we will have to set $\delta=-u+\tau$ with fixed $\tau$. This is performed in Section~\ref{sec:limitTransitionBR}, where we recover the Baik--Rains distribution.
\end{rem}

As for the finite $N$ case, the main theorem in the $N\to\infty$ limit requires definitions of various objects. Let us define the functions
\begin{equation}
\begin{aligned}
\mathpzc{f}^{-\delta,u}(X)&=e^{-\frac{\delta^3}{3} - \delta^2 u + \delta X}, \\
\mathpzc{e}^{\delta, u} (S)&= -\int\limits_{\zcd\, {}_\delta} \frac{d \zeta}{2\pi\I} \frac{e^{\frac{\zeta^3}{3} + \zeta^2 u - \zeta S} }{e^{\frac{\delta^3}{3} + \delta^2 u - \delta S}} \frac{1}{(\zeta - \delta)^2}, \\
\mathpzc{j}^{\delta, u}(S,X)&=\left[\frac{\sinh{\delta(X-S)}}{\delta}+ (X-S)e^{\delta(X-S)}\right]\mathpzc{f}^{-\delta,-u}(S)
\end{aligned}
\end{equation}
as well as
\begin{alignat}{2}
\mathpzc{g}_1^{\delta, u} (X) &= \int\limits_{{ }_0\zcd} \frac{d \zeta}{2\pi\I} e^{\frac{\zeta^3}{3} - \zeta^2 u - \zeta X} \frac{\zeta + \delta}{2 \zeta},
\quad &  \mathpzc{g}_2^{\delta, u} (X) &= \int\limits_{\zcd\, {}_\delta} \frac{d \zeta}{2\pi\I} e^{\frac{\zeta^3}{3} + \zeta^2 u - \zeta X} \frac{1}{\zeta - \delta}, \nonumber \\
\mathpzc{g}_3^{\delta, u} (X) &= \int\limits_{{}_{-\delta}\zcd} \frac{d \zeta}{2\pi\I} e^{\frac{\zeta^3}{3} - \zeta^2 u - \zeta X} \frac{1}{\zeta+\delta},
\quad & \mathpzc{g}_4^{\delta, u} (X) &= \int\limits_{\zcd\, {}_{\pm\delta}} \frac{d \zeta}{2\pi\I} e^{\frac{\zeta^3}{3} + \zeta^2 u - \zeta X} \frac{2\zeta}{(\zeta-\delta)(\zeta + \delta)^2}.
\end{alignat}

\noindent The limit of $\overline{\mathsf{K}}$ is the following \emph{anti-symmetric} kernel $\overline{\mathcal{A}}$:
\begin{equation}\label{eq:227}
\begin{aligned}
\overline{\mathcal{A}}_{11} (X, Y) &= - \int\limits_{{ }_{0}\zcd } \frac{d \zeta}{2\pi\I} \int\limits_{\wcu\, {}_{0,\zeta}} \frac{d \omega}{2\pi\I}\frac{ e^{\frac{\zeta^3}{3} - \zeta^2 u - \zeta X} }{ e^{\frac{\omega^3}{3} + \omega^2 u - \omega Y} } (\zeta - \delta)(\omega + \delta) \frac{\zeta + \omega}{4 \zeta \omega (\zeta - \omega)},\\
\overline{\mathcal{A}}_{12} (X, Y) &= -\int\limits_{{}_{0}{\zcd}} \frac{d \zeta}{2\pi\I} \int\limits_{{}_{\delta}\wcu\, {}_\zeta} \frac{d \omega}{2\pi\I} \frac{ e^{\frac{\zeta^3}{3} - \zeta^2 u - \zeta X} }{ e^{\frac{\omega^3}{3} - \omega^2 u - \omega Y} } \frac{\zeta-\delta}{\omega-\delta} \frac{\zeta+\omega}{2 \zeta (\zeta-\omega)}\\
&= -\overline{\mathcal{A}}_{21} (Y, X),\\
\overline{\mathcal{A}}_{22} (X, Y) &= \mathcal{E}(X,Y) + \int \frac{d \zeta}{2\pi\I} \int \frac{d \omega}{2\pi\I} \frac{ e^{\frac{\zeta^3}{3} + \zeta^2 u - \zeta X} }{ e^{\frac{\omega^3}{3} - \omega^2 u - \omega Y} } \frac{1}{\zeta - \omega} \left(\frac{1}{ \zeta + \delta}+\frac{1}{\omega-\delta}\right)
\end{aligned}
\end{equation}
where in $\overline{\mathcal{A}}_{22}$ the integration contours, for $(\zeta,\omega)$, are $\zcd\, {}_{-\delta}\times \wcu\, {}_{\zeta}$ for the term $1/(\zeta+\delta)$, and $\zcd\, {}\times{}_{\delta}\wcu\, {}_\zeta$ for the term $1/(\omega-\delta)$. We have denoted $\mathcal{E} = \mathcal{E}_0+ \mathcal{E}_1$ with
\begin{equation}
  \begin{split}
\mathcal{E}_0 (X, Y) &= -\sgn(X - Y) e^{\delta |X - Y| + 2 \delta^2 u},\\
\mathcal{E}_1 (X, Y) &= -\sgn(X - Y) \int\limits_{{}_{\pm\delta}\zcd} \frac{d \zeta}{2\pi\I} e^{-\zeta |X - Y| + 2 \zeta^2 u}
\frac{2\zeta}{\zeta^2-\delta^2}.
  \end{split}
\end{equation}
Finally, we also set
\begin{equation}
\begin{aligned}
\widetilde{\mathcal{A}}_{12} (X, Y) &= -\int\limits_{{}_{0}{\zcd }} \frac{d \zeta}{2\pi\I} \int\limits_{\wcu\, {}_{\delta,\zeta}} \frac{d \omega}{2\pi\I} \frac{ e^{\frac{\zeta^3}{3} - \zeta^2 u - \zeta X} }{ e^{\frac{\omega^3}{3} - \omega^2 u -\omega Y} } \frac{\zeta-\delta}{\omega-\delta} \frac{\zeta+\omega}{2 \zeta (\zeta-\omega)},\\
\widetilde{\mathcal{A}}_{22} (X, Y) &= \int\limits_{\zcd\, {}_{-\delta}} \frac{d \zeta}{2\pi\I} \int\limits_{\wcu\, {}_{\delta,\zeta}} \frac{d \omega}{2\pi\I} \frac{ e^{\frac{\zeta^3}{3} + \zeta^2 u - \zeta X} }{ e^{\frac{\omega^3}{3} - \omega^2 u - \omega Y} } \frac{1}{(\zeta + \delta)(\omega - \delta)} \frac{\zeta + \omega}{\zeta - \omega}
\end{aligned}
\end{equation}
and
\begin{equation}
\begin{aligned}
\mathpzc{h}^{\delta,u}_1(Y) &= \int_{S}^\infty dV \widetilde{\mathcal{A}}_{22}(Y,V) \mathpzc{f}^{-\delta, u}(V) +\int_{S}^\infty dV \mathcal{E}_1(Y,V)\mathpzc{f}^{-\delta,u}(V) -\mathpzc{g}_4^{\delta,u}(Y) - \mathpzc{j}^{\delta, u} (S,Y),\\
\mathpzc{h}^{\delta, u}_2(Y)&= \int_S^\infty dV \widetilde{\mathcal{A}}_{12}(Y,V) \mathpzc{f}^{-\delta, u}(V)+ \mathpzc{g}_3^{\delta,u}(Y).
\end{aligned}
\end{equation}

Then, the limiting distribution of the rescaled last passage percolation in the stationary case is given as follows.

\begin{thm} \label{thm:main_thm_asymp}
 Let $\delta \in \R$, $u > 0$ be parameters. Consider the stationary LPP time $L_{N,N-n}$ from $(1, 1)$ to $(N, N-n)$ and the scaling
 \begin{equation}
  n =  u 2^{5/3} N^{2/3}, \quad \alpha =  \delta 2^{-4/3} N^{-1/3}.
 \end{equation}
We have that
 \begin{equation}
  \lim_{N \to \infty} \Pb \left( \frac{L_{N,N-n} - 4N + 4 u (2N)^{2/3}}{2^{4/3}N^{1/3}}  \leq S\right) = F^{(\delta, u)}_{0,\,\mathrm{half}}(S)
 \end{equation}
with
 \begin{equation}\label{eq:final_dist_asymp}
 \begin{split}
  F^{(\delta, u)}_{0,\,\mathrm{half}} (S) = \partial_S \Bigg\{ \pf(J - \overline{\mathcal{A}}) \cdot \Bigg[ \mathpzc{e}^{\delta, u}(S) - \braket{-\mathpzc{g}_1^{\delta, u} \quad \mathpzc{g}_2^{\delta, u} } { (\Id-J^{-1}\overline{\mathcal{A}})^{-1} \begin{pmatrix} -\mathpzc{h}^{\delta, u}_1 \\ \mathpzc{h}^{\delta, u}_2 \end{pmatrix} } \Bigg] \Bigg\}
  \end{split}
 \end{equation}
 where the Fredholm Pfaffian is taken over $L^2 (S, \infty) \times L^2 (S, \infty)$ and where $J = \left( \begin{smallmatrix} 0 & 1 \\ -1 & 0 \end{smallmatrix} \right)$.
\end{thm}

\begin{rem}
A numerical evaluation of the term inside the $\partial_S$ does not require the inversion of the kernel $J^{-1}\overline{\mathcal{A}}$, since it can be written as a linear combination of Fredholm Pfaffians by Lemma~\ref{lem:3.19}. This trick has been already observed by Imamura--Sasamoto~\cite{IS12} in the context of the stationary KPZ equation. Let $\xi_{(\delta, u)} \sim F^{(\delta, u)}_{0,\,\mathrm{half}}(S)=\partial_S G(S)$. Then the computation of moments can be made without taking $\partial_S$. Indeed, for the expected value, stationarity implies
\begin{equation}
\E(L_{N,N-n})=N/(1/2-\alpha)+(N-n)/(1/2+\alpha)\quad \Rightarrow\quad \E(\xi_{(\delta, u)})=\delta(2u+\delta)
\end{equation}
while for all $\ell\geq 2$, integration by parts gives
\begin{equation}
\E(\xi_{(\delta, u)}^\ell)=\ell(\ell-1) \int_{\R_+} S^{\ell-2} [G(S)-S] dS + \ell(\ell-1)\int_{\R_-} S^{\ell-2} G(S) dS.
\end{equation}
\end{rem}

\begin{rem}
The origin of the terms $\widetilde{\cal A}_{12}$ and $\widetilde{\cal A}_{22}$ in lieu of $\overline{\cal A}_{12}$ and $\overline{\cal A}_{22}$ stems from the fact that for $\delta \geq 0$, the product of the latter with $\mathpzc{f}^{-\delta,u}$ is not well-defined. However, for $\delta<0$ this is not the case and so using the tilde kernels is not necessary. For that reason and for $\delta<0$, we can thus simplify the expression of the $\mathpzc{h}_k^{\delta, u}$ entering equation \eqref{eq:final_dist_asymp} as follows.
\end{rem}

\begin{lem}\label{lem:SimplificationBR}
For $\delta<0$, the formula \eqref{eq:final_dist_asymp} holds also for $\mathpzc{h}_k^{\delta, u}$ replaced by $\tilde{\mathpzc{h}}_k^{\delta, u}$, $k=1,2$, with the latter defined by
\begin{equation}
\begin{aligned}
\tilde{\mathpzc{h}}^{\delta,u}_1(Y) &= \int_{S}^\infty dV \overline{\mathcal{A}}_{22}(Y,V) \mathpzc{f}^{-\delta, u}(V) -\mathpzc{g}_4^{\delta,u}(Y),\\
\tilde{\mathpzc{h}}^{\delta, u}_2(Y)&= \int_S^\infty dV \overline{\mathcal{A}}_{12}(Y,V) \mathpzc{f}^{-\delta, u}(V)+ \mathpzc{g}_3^{\delta,u}(Y).
\end{aligned}
\end{equation}
\end{lem}

\begin{rem} \label{rem:a_bar}
 The kernel $\overline{\mathcal{A}}$ with parameter $\delta$ corresponds to the limiting crossover kernel from~\cite{BBCS17, BBNV18, SI03} after matching the $\delta$ parameter with notations therein. This arises when considering LPP in half-space with boundary term only along the diagonal. The respective distribution given by $\pf(J - \overline{\mathcal{A}})_{L^2(S,\infty) \times L^2(S,\infty)}$, for $u=0$, is the interpolating GOE to GSE distribution of Baik and Rains~\cite{BR99b}. Similar to Remark~\ref{rem:k_bar}, the Pfaffian $\pf(J - \overline{\mathcal{A}})_{L^2(S,\infty) \times L^2(S,\infty)}$ is in $(0, 1)$ for any fixed $S$.
\end{rem}

\subsection{Limit transition to the Baik--Rains distribution}\label{sec:limitTransitionBR}

One might wonder if the Baik--Rains distribution with parameter $\tau$ arises in some appropriate limit\footnote{The Baik--Rains distribution away from the characteristic line is given in~\cite{BR00}, formula (3.35).}. The answer is affirmative. Heuristically, when $u$ is increasing (the end-point of the LPP is moving away from the diagonal), the maximizer $\pi^{\rm max}$ of the LPP will eventually not visit anymore the diagonal at distance of much larger than $N^{2/3}$, since $\pi^{\rm max}$ stays in a $\Or(N^{2/3})$ distance from the straight line from the origin to the end point. Thus in the $u\to\infty$ limit, $\pi^{\rm max}$ will visit the diagonal only in a $\Or(N^{2/3})$-neighborhood of the origin. This geometry is similar to the full-space problem and consequently one might expect to recover the Baik--Rains distribution. This is the result presented in this section.

To state the theorem, we need to introduce a few functions and a limiting kernel. Let us define\footnote{Here we use notation as in~\cite{FS05a,BFP09} with one exception: for the functions involved, we keep their contour integral representations instead of rewriting them in terms of (integrals of) Airy and exponential functions. The original expression in~\cite{BR00} looks quite different, but as is shown in~\cite[Appendix A]{FS05a}, gives the same distribution.}
\begin{equation}
\begin{aligned}
{\cal R}_{\tau}(s)&=-e^{-\frac23 \tau^3 - s \tau} \int\limits_{\zcd\, {}_{-\tau}} \frac{dz}{2 \pi \I} e^{\frac{z^3}{3} - z (s+\tau^2)} \frac{1}{(z+\tau)^2},\\
\Psi_{\tau}(x)&=\int\limits_{\zcd\, {}_{-\tau}} \frac{dz}{2 \pi \I} e^{\frac{z^3}{3} - z (x+\tau^2)} \frac{1}{z+\tau},\\
\Phi_{\tau}(y)&=e^{-\frac23 \tau^3 -s \tau} \int\limits_{\zcd} \frac{d z}{2\pi\I} \int\limits_{{}_{\tau}\wcu\, {}_{z}} \frac{d w}{2\pi\I} \frac{ e^{\frac{z^3}{3} - z (y+\tau^2)} }{ e^{\frac{w^3}{3} - w (s+\tau^2)} } \frac{1}{(z-w) (w - \tau)}
\end{aligned}
\end{equation}
and the shifted Airy kernel\footnote{The reason for the minus sign in front is that our $z$ contour $\zcd$ is oriented downwards.}
\begin{equation}
\mathpzc{K}_{\ \rm Ai,\tau} (x, y) = -\int\limits_{\zcd} \frac{d z}{2\pi\I} \int\limits_{ \wcu\, {}_z} \frac{d w}{2\pi\I} \frac{ e^{\frac{z^3}{3}- z (x+\tau^2)} }{ e^{\frac{w^3}{3} - w (y+\tau^2)} } \frac{1}{ (z-w) }.
\end{equation}

\begin{thm}\label{thm:LimitToBR}
Let $S=s+\delta(2u+\delta)$ and $u+\delta=\tau$ fixed. Then we have:
\begin{equation}
\lim_{u\to\infty} F^{(\delta, u)}_{0,\,\mathrm{half}} (S) = F_{{\rm BR},\tau}(s)
\end{equation}
where $F_{{\rm BR},\tau}(s)$ is the extended Baik--Rains distribution, defined by
\begin{equation}
F_{{\rm BR},\tau}(s) = \partial_s \left[ F_{\rm GUE}(s+\tau^2) \cdot \left(
{\cal R}_\tau - \braket{ \Psi_\tau } {(\Id-\mathpzc{K}_{\ \rm Ai,\tau})^{-1} \Phi_\tau }
\right) \right]
\end{equation}
with the operators in the scalar product being on $L^2(s,\infty)$ and with $F_{\rm GUE}$ the Tracy--Widom distribution.
\end{thm}

\section{Finite system stationary model: proof of Theorem~\ref{thm:main_thm_fin}} \label{sec:proofFiniteN}

\subsection{The integrable model} \label{sec:int}

In this section we consider the slightly modified LPP model with weights
\begin{equation} \label{eq:int_wts}
  \tilde\omega_{i, j} = \begin{cases}
    \mathrm{Exp}\left( \frac{1}{2} + \alpha \right), & i=j>1,\\
    \mathrm{Exp}\left( \frac{1}{2} + \beta \right), & j=1, i>1, \\
    \mathrm{Exp}\left( \alpha + \beta \right), & i=j=1, \\
    \mathrm{Exp}(1), &\text{otherwise}
  \end{cases}
\end{equation}
where $\alpha \in (-1/2, 1/2), \beta \in (-1/2,1/2)$ are parameters satisfying $\alpha + \beta > 0$, see  Figure~\ref{fig:exp_lpp} for an illustration. We denote by $\Lpf$ the LPP with weights $\tilde\omega$, from $(1,1)$ to the point $(N,N-n)$.
\begin{figure}[t!]
  \centering
   \includegraphics[height=5.5cm]{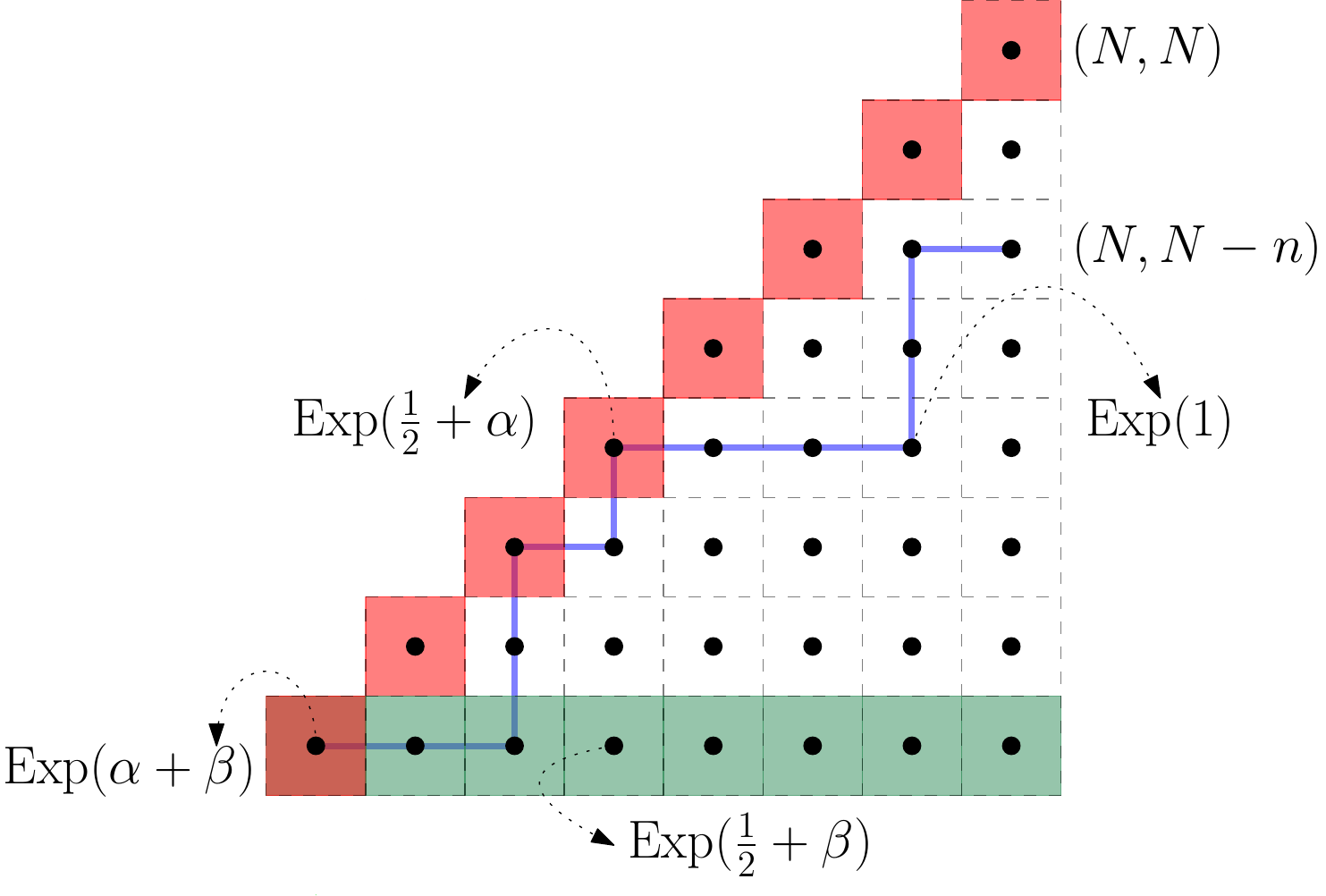}
   \caption{A possible LPP path (polymer) from $(1, 1)$ to $(N, N-n)$ for $(N, n) = (8, 2)$ in the integrable case. The dots are independent exponential random variables: $\mathrm{Exp}(\alpha + \beta)$ at the origin, $\mathrm{Exp}(\tfrac12+\alpha)$ (respectively $\mathrm{Exp}(\tfrac12 + \beta)$) on rest of the diagonal (respectively the bottom line), and $\mathrm{Exp}(1)$ everywhere else in the bulk.}
   \label{fig:exp_lpp}
\end{figure}

For the case of $\beta>0$, it has been shown that the distribution of $\Lpf$ is given by a Fredholm Pfaffian. The next theorem, which is the starting point of our analysis, is obtained as a limit of a geometric version of the model, whose correlation kernel was obtained in the work of Baik--Barraquand--Corwin--Suidan~\cite{BBCS17} and Betea--Bouttier--Nejjar--Vuleti\'c~\cite{BBNV18}. The limit from geometric to exponential in the $\beta=1/2$ case was carried out in~\cite{BBCS17} using steepest descent analysis (Laplace method). This is actually not strictly necessary. We shall see this in the proof of Theorem~\ref{thm:exp_corr}, presented in Appendix~\ref{sec:geom_wts}.
\begin{thm} \label{thm:exp_corr}
Let $\beta\in (0,1/2)$, $\alpha\in (-1/2,1/2)$ with $\alpha+\beta>0$. Then, for $s\in\R_+$,
  \begin{equation}
  \Pb (\Lpf \leq s) = \pf (J - K)_{L^2 (s, \infty) \times L^2 (s, \infty)}
  \end{equation}
where $K=K(x, y)$ is the following $2 \times 2$ matrix kernel:
  \begin{align} \label{eq:kernel}
   K_{11}(x,y) =& - \oint \frac{dz}{2\pi\I} \oint\frac{dw}{2\pi\I}\frac{\Phi(x,z)}{\Phi(y,w)}\left[(\tfrac12-z)(\tfrac12+w)\right]^n\frac{(z+\beta)(w-\beta)}{(z-\beta)(w+\beta)}\frac{(z+\alpha)(w-\alpha)(z+w)}{4zw(z-w)}, \nonumber \\
   K_{12}(x,y) =& - \oint \frac{dz}{2\pi\I} \oint \frac{dw}{2\pi\I}\frac{\Phi(x,z)}{\Phi(y,w)}\left[\frac{\tfrac12-z}{\tfrac12-w}\right]^n\frac{z+\alpha}{w+\alpha}\frac{z+\beta}{z-\beta}\frac{w-\beta}{w+\beta}\frac{z+w}{2z(z-w)} \nonumber \\
               =& -K_{21} (y,x),\\
   K_{22}(x,y) =& \oint \frac{dz}{2\pi\I} \oint \frac{dw}{2\pi\I} \frac{\Phi(x,z)}{\Phi(y,w)} \frac{1}{\left[(\tfrac12+z)(\tfrac12-w)\right]^n}\frac{1}{(z-\alpha)(w+\alpha)} \frac{z+\beta}{z-\beta}\frac{w-\beta}{w+\beta}\frac{z+w}{z-w} \nonumber \\
   & + \tilde \varepsilon(x, y). \nonumber
  \end{align}
The integration contours are the union of the following:
\begin{itemize}
\item for $K_{11}$, $(z,w) \in \Gamma_{1/2}\times\Gamma_{-1/2}$ and $(z,w) \in \Gamma_{1/2}\times\Gamma_{-\beta}$ and $(z,w) \in \Gamma_{\beta}\times\Gamma_{-1/2}$;
\item for $K_{12}$, $(z,w) \in \Gamma_{1/2}\times\Gamma_{-1/2,-\alpha,-\beta}$ and $(z,w) \in \Gamma_\beta\times\Gamma_{-1/2,-\alpha}$;
\item for $K_{22}$, $(z,w) \in \Gamma_{1/2,\alpha,\beta}\times\Gamma_{-1/2}$ and $(z,w) \in \Gamma_{1/2,\beta}\times\Gamma_{-\alpha}$ and $(z,w) \in \Gamma_{1/2,\alpha}\times\Gamma_{-\beta}$.
\end{itemize}
Here we have denoted
\begin{equation} \label{eq:phi}
  \Phi(x, z)= e^{-xz} \phi (z), \quad \textrm{with } \phi(z)= \left[ \frac{ \tfrac12 + z } { \tfrac12-z }\right]^{N-1}
\end{equation}
and $\tilde \varepsilon = \varepsilon_1 + \varepsilon_2$ with
\begin{equation}
 \varepsilon_1(x, y) =  - \sgn(x-y)\oint\limits_{\Gamma_{1/2}} \frac{dz}{2\pi\I} \frac{2z e^{-z |x-y|}}{(z^2-\alpha^2) \left( \frac{1}{4} - z^2 \right)^n}, \quad \varepsilon_2(x,y) =  - \sgn(x-y) \frac{ e^{-\alpha |x-y|} } {\left( \frac{1}{4} - \alpha^2 \right)^n}.
\end{equation}
\end{thm}

\begin{rem}
  We remark the following trivial but useful identities
  \begin{equation}\label{eq:phi_symm}
   \phi(-z)=\phi(z)^{-1},\quad \Phi(x, -z)=\Phi(x, z)^{-1}
  \end{equation}
  which we shall use throughout many times without explicit reference. Further, note that for $n=0$, $\tilde \varepsilon$ simplifies to $\tilde \varepsilon(x, y) = -\sgn(x-y) e^{-\alpha |x-y|}$ since the pole at $z=1/2$ vanishes.
\end{rem}

\subsection{From integrable to stationary} \label{sec:int_to_stat}

\subsubsection{Shift argument} \label{sec:shift}

To recover the desired distribution, we need to remove $\tilde \omega_{1,1}$ and then take the $\beta \to -\alpha$ limit. The former is achieved by a standard shift argument, used already in the full-space stationary LPP problem~\cite{BR00, FS05a, BFP09, SI04}\footnote{Baik--Rains~\cite{BR00} treat the Poisson case instead of the exponential one but the shift argument is similar.}. We present the short proof for completeness.

We recall that $\Lpf$ denotes the LPP time for the random variables $\tilde{\omega}_{i,j}$ of~\eqref{eq:int_wts}. Denote by $\widetilde L_{N,N-n} = \Lpf - \tilde\omega_{1,1}$ and recall that $L_{N,N-n}$ is the $\beta\to-\alpha$ limit of $\widetilde L_{N,N-n}$. The shift argument is captured by the following lemma.
\begin{lem} \label{lem:shift}
Let $\alpha,\beta\in(-1/2,1/2)$ with $\alpha+\beta>0$. Then
\begin{equation}\label{eq:shift3}
 \left(1+\frac{1}{\alpha+\beta}\partial_s\right)\Pb(\Lpf\leq s)=\Pb(\widetilde L_{N,N-n}\leq s).
\end{equation}
\end{lem}
\begin{proof}
Due to the independence of $\widetilde L_{N,N-n}$ and $\tilde\omega_{1,1}$, we have
\begin{equation}
\Pb(\Lpf\leq s)=\Pb(\tilde\omega_{1,1}+\widetilde L_{N,N-n}\leq s)=(\alpha+\beta)\int_0^\infty d\lambda e^{-\lambda(\alpha+\beta)}\Pb(\widetilde L_{N,N-n}\leq s-\lambda).
\end{equation}
Performing the Laplace transform and the change of variables $s-\lambda=u$, we obtain
\begin{equation}
\begin{aligned}\label{eq:shift1}
 \int_0^\infty ds e^{-ts}\Pb(\Lpf\leq s) &= \int_0^\infty du\int_0^\infty d\lambda (\alpha+\beta)e^{-ts}e^{-\lambda(\alpha+\beta)}\Pb(\widetilde L_{N,N-n}\leq s-\lambda)\\
 &= (\alpha+\beta)\int_0^\infty d\lambda e^{-\lambda(\alpha+\beta+t)} \int_0^\infty du e^{-tu}\Pb(L\leq u).
\end{aligned}
\end{equation}
Computing the first integral on the right-hand side of \eqref{eq:shift1}, and then integrating by parts, we obtain
\begin{equation}
\begin{aligned}\label{eq:shift2}
 \int_0^\infty ds e^{-ts}\Pb(\widetilde L_{N,N-n} \leq s) =& \left(1+\frac{t}{\alpha+\beta}\right)\int_0^\infty ds e^{-ts}\Pb(\Lpf\leq s)\\
   =& \int_0^\infty ds e^{-ts}\Pb(\Lpf\leq s)+ \frac{1}{\alpha+\beta} e^{-ts}\Pb(\Lpf\leq s)\big|_0^\infty \\
  &+\frac{1}{\alpha+\beta} \int_0^\infty ds e^{-ts}\frac{d}{ds}\Pb(\Lpf\leq s)
\end{aligned}
\end{equation}
which gives \eqref{eq:shift3} since the second term is $0$.
\end{proof}

\subsubsection{Kernel decomposition} \label{sec:kernel_decomp}

From Lemma~\ref{lem:shift}, we need to find a decomposition of $\frac{1}{\alpha+\beta} \Pb(\Lpf\leq s)$ which has a well-defined limit as $\alpha+\beta\to 0$. For that purpose, we first decompose the kernel by separating the contributions of the different poles in a way that will be convenient for future computations. The result will be given in terms of the following functions:
\begin{alignat}{2} \label{eq:fg_def}
 f_{+}^{\beta}(x)& = \Phi(x, \beta)\left( \tfrac12-\beta \right)^n,
 & f_{-}^{\beta}(x)&= \frac{\Phi(x, \beta)}{\left( \tfrac12+\beta \right)^n}, \nonumber \\
 g_1(x)& = \oint\limits_{\Gamma_{1/2}} \frac{dz}{2\pi\I} \Phi (x,z)\left( \tfrac12-z \right)^n\frac{z+\alpha}{2z},
  & g_2(x)& = \!\!\! \oint\limits_{\Gamma_{1/2,\alpha}} \!\!\!\frac{dz}{2\pi\I} \frac{\Phi(x,z)}{\left( \tfrac12+z \right)^n}\frac{1}{z-\alpha}, \\
g_3(x)&= \oint\limits_{\Gamma_{1/2}} \frac{dz}{2\pi\I} \Phi (x,z)\left( \tfrac12-z \right)^n\frac{1}{z-\beta},
 & g_4(x)&= \!\!\!\!\!\! \oint\limits_{\Gamma_{1/2,\pm \alpha,\beta}} \!\!\!\!\!\! \frac{dz}{2\pi\I} \frac{\Phi(x,z)}{\left( \tfrac12+z \right)^n}\frac{2z}{(z-\alpha)(z+\alpha)(z-\beta)}, \nonumber \\
 g_5(x)& = \oint\limits_{\Gamma_{1/2}} \frac{dz}{2\pi\I} \Phi(x,z)\left( \tfrac12-z \right)^n \frac{(z-\alpha)(z+\beta)}{2z(z-\beta)},
& &  g_6(x) = \oint\limits_{\Gamma_{1/2}} \frac{dz}{2\pi\I} \frac{\Phi(x,z)}{\left( \tfrac12+z \right)^n} \frac{z+\beta}{(z+\alpha) (z-\beta)}. \nonumber
\end{alignat}

With these notations we can now write the kernel decomposition used later.
\begin{prop} \label{prop:decomposition} Let $\alpha \in (-1/2,1/2)$, $\beta\in (0,1/2)$ with $\alpha+\beta>0$ and $\alpha\neq\beta$. Then the kernel $K$ splits as
 \begin{equation}
  K = \overline{K} + (\alpha+\beta) R
 \end{equation}
 where
 \begin{equation} \label{eq:h_bar_def}
  \begin{split}
   \overline{K}_{11}(x,y) =& - \!\!\! \oint\limits_{\Gamma_{1/2}} \!\!\! \frac{dz}{2\pi\I} \!\!\! \oint\limits_{\Gamma_{-1/2}} \!\!\! \frac{dw}{2\pi\I}\frac{\Phi(x,z)}{\Phi(y,w)}\left[(\tfrac12-z)(\tfrac12+w)\right]^n\frac{(z+\beta)(w-\beta)}{(z-\beta)(w+\beta)}\frac{(z+\alpha)(w-\alpha)(z+w)}{4zw(z-w)}, \\
   \overline{K}_{12}(x,y) =& - \!\!\! \oint\limits_{\Gamma_{1/2}} \!\!\! \frac{dz}{2\pi\I} \!\!\! \oint\limits_{\Gamma_{-1/2,-\alpha,-\beta}} \!\!\!\!\!\!\!\!\!\! \frac{dw}{2\pi\I}\frac{\Phi(x,z)}{\Phi(y,w)} \left[ \frac{\tfrac12-z}{\tfrac12-w}\right]^n\frac{z+\alpha}{w+\alpha}\frac{z+\beta}{z-\beta}\frac{w-\beta}{w+\beta}\frac{z+w}{2z(z-w)} \\
               =& -\overline{K}_{21} (y,x),\\
   \overline{K}_{22}(x,y) =&\ \tilde\e(x,y)+\oint \frac{dz}{2\pi\I}\oint\frac{dw}{2\pi\I} \frac{\Phi(x,z)}{\Phi(y,w)} \frac{1}{\left[(\tfrac12+z)(\tfrac12-w)\right]^n}\frac{1}{(z-\alpha)(w+\alpha)} \frac{z+\beta}{z-\beta}\frac{w-\beta}{w+\beta}\frac{z+w}{z-w}
  \end{split}
 \end{equation}
and where the integration contours for $\overline{K}_{22}$ are, for $(z,w)$, the union of $\Gamma_{1/2,\alpha,\beta}\times\Gamma_{-1/2}$, $\Gamma_{1/2,\beta}\times \Gamma_{-\alpha}$, and $\Gamma_{1/2,\alpha}\times\Gamma_{-\beta}$.

The operator $R$ is of rank two and given by
 \begin{equation}
  R=\begin{pmatrix}
        \ketbra{g_1}{f_{+}^{\beta}} - \ketbra{f_{+}^{\beta}}{g_1}  & \ketbra{f_{+}^{\beta}} {g_2} \\
       - \ketbra{g_2}{f_{+}^{\beta}}  & 0
 \end{pmatrix}.
 \end{equation}
\end{prop}
\begin{proof}
The proof amounts to residue computations. In $K_{11}$ the terms coming into $R$ are the residue at $(z=\beta,w=-1/2)$ and at $(z=1/2,w=-\beta)$. For $K_{12}$, the terms in $R$ are the residues from $(z=\beta,w=-1/2)$ and $(z=\beta,w=-\alpha)$.
\end{proof}

\begin{rem}
The condition $\alpha \neq \beta$ is not restrictive since in the desired limit we deal with $\beta \to -\alpha$. Even for $\alpha=0$, at this stage, we need to have $\alpha+\beta=\beta>0$. Of course, the kernel is well-defined also for $\alpha=\beta$. It is enough to take the contours for $(z,w)$ as $\Gamma_{1/2,\alpha=\beta}\times\Gamma_{-1/2,-\alpha=-\beta}$, which can be taken non-intersecting as $\beta>0$.
\end{rem}

\subsubsection{Decomposition of the rank-two perturbation $R$} \label{sec:R_decomp}

Since we are going to work with Fredholm determinants instead of Pfaffians for a while, we define
\begin{equation}
 G = J^{-1} K, \quad \overline{G} = J^{-1} \overline{K}, \quad T = J^{-1} R
\end{equation}
where $J(x,y)=\delta_{x,y} \left( \begin{smallmatrix} 0 & 1 \\ -1 & 0 \end{smallmatrix} \right)$ so that, for instance, the matrix kernel $G$ is given by
\begin{equation}
 G (x,y)=\begin{pmatrix}
             -K_{21}(x,y) & -K_{22}(x,y) \\ K_{11}(x,y) & K_{12}(x,y)
         \end{pmatrix}
\end{equation}
while $T$ is given by
\begin{equation}
 T = \ketbra{X_1}{Y_1} + \ketbra{X_2}{Y_2}
\end{equation}
with
\begin{alignat}{2} \label{eq:S_fact}
  X_1 &= \ket{ \begin{array}{c} g_2 \\ g_1 \end{array} }, \quad & X_2 &= \ket{ \begin{array}{c} 0 \\ f_{+}^\beta \end{array} }, \\
  Y_1 &= \bra{f_{+}^\beta \quad 0}, \quad & Y_2 &= \bra{-g_1 \quad g_2}.
\end{alignat}
Since all Fredholm determinants/Pfaffians as well as (scalar) products will be on $L^2(s, \infty) \times L^2(s, \infty)$, we consider all operators to be defined on this space and omit it from the notation for brevity. Recall also that $\pf (J-K) = \sqrt{\det(\Id-G)}$.

From the shift argument Lemma~\ref{lem:shift}, we need to determine the $\alpha+\beta\to 0$ limit of
\begin{equation}
\frac{1}{\alpha+\beta}\pf(J-K)=\frac{1}{(\alpha+\beta)}\sqrt{\det(\Id-G)}
\end{equation}
where
\begin{equation}
 \det(\Id - G) = \det(\Id - \overline{G}) \cdot \det (\Id - (\alpha+\beta) (\Id-\overline{G})^{-1} T).
\end{equation}
Setting
\begin{equation}
Z_i= (\Id - \overline{G})^{-1} X_i, \quad i=1,2,
\end{equation}
we get
\begin{equation}\label{eq:F_det}
\begin{aligned}
(\Id - (\alpha+\beta) (\Id-\overline{G})^{-1} T)  &= \det\left(\Id-(\alpha+\beta)
                \begin{pmatrix}
                    \braket{Y_1}{Z_1} & \braket{Y_2}{Z_1}\\
                    \braket{Y_1}{Z_2} & \braket{Y_2}{Z_2}
                \end{pmatrix}
                \right) \\
    &= \det(\Id-\overline{G}) (1-(\alpha+\beta)\braket{Y_2}{Z_2})^2.
\end{aligned}
\end{equation}
In~\eqref{eq:F_det}, the first equality is just a rewriting and holds in any inner product space and for any vectors $Y_i, Z_i, i=1,2$, see for instance~\cite{TW96} for more details, while for the second we used the equalities
\begin{equation}
  \braket{Y_1}{Z_2} = \braket{Y_2}{Z_1} = 0, \quad \braket{Y_1}{Z_1} = \braket{Y_2}{Z_2}
\end{equation}
proven in Appendix~\ref{sec:det_comp}.

Summarizing, we need to determine the $\beta\to -\alpha$ limit of
\begin{equation} \label{eq:lim}
\begin{aligned}
\frac{1}{\alpha+\beta}\pf(J-K) &= \pf(J-\overline{K}) \left(\frac{1}{\alpha+\beta}- \braket{Y_2}{Z_2}\right)\\
&=\pf(J-\overline{K}) \left(\frac{1}{\alpha+\beta}- \braket{Y_2}{X_2} - \braket{Y_2}{(\Id-\overline G)^{-1}\overline G X_2}\right).
\end{aligned}
\end{equation}

\subsection{Analytic continuation} \label{sec:analytic_continuation}

Recall that we started with our kernels defined for $\beta>0$ only. Now we have to deal with the analyticity of the right-hand side of~\eqref{eq:lim} and determine the desired limit.

Throughout, when we say that a function is analytic in $\alpha,\beta\in (-1/2,1/2)$, we mean that for any $0 < \epsilon \ll 1$, the function is analytic in $\alpha, \beta \in [-1/2+\epsilon,1/2-\epsilon]$.

\begin{rem} \label{rem:mathsf}
  Hereinafter, as anticipated in the introduction, we will denote, in up-right sans-serif $\mathsf{font}$, the limits as $\beta \to - \alpha$ of the various kernels and functions we use and which depend explicitly on $\beta$. The ones that are independent of $\beta$ we leave unchanged. Thus by definition:
  \begin{equation}
    (\mathsf{g}_3, \mathsf{g}_4, \mathsf{g}_5, \mathsf{g}_6) = \lim_{\beta \to -\alpha} (g_3, g_4, g_5, g_6), \quad (\overline{\mathsf{K}}, \widetilde{\mathsf{K}}, \widehat{\mathsf{G}}) = \lim_{\beta \to -\alpha} (\overline{K}, \widetilde{K}, \widehat{G})
  \end{equation}
  where the $g$'s are defined in~\eqref{eq:fg_def}, $\overline{K}$ in~\eqref{eq:h_bar_def}, and $\widetilde{K}, \widehat{G}$ will be defined below.
\end{rem}

\subsubsection{Analyticity of the Fredholm Pfaffian} \label{sec:analyticity_Fredholm}

\begin{lem} \label{lem:Kbar}
  The kernel $\overline{K}$ is analytic for $\alpha, \beta \in (-1/2, 1/2)$. The limit kernel $\overline{\mathsf{K}}=\lim_{\beta\to-\alpha} \overline{K}$
   has the following entries:
    \begin{equation}\label{eq:336}
  \begin{split}
   \overline{\mathsf{K}}_{11}(x,y) =& - \oint\limits_{\Gamma_{1/2}} \frac{dz}{2\pi\I} \oint\limits_{\Gamma_{-1/2}} \frac{dw}{2\pi\I}\frac{\Phi(x,z)}{\Phi(y,w)}\left[(\tfrac12-z)(\tfrac12+w)\right]^n\frac{(z-\alpha)(w+\alpha)(z+w)}{4zw(z-w)}, \\
   \overline{\mathsf{K}}_{12}(x,y) =& - \!\!\! \oint\limits_{\Gamma_{1/2}}  \frac{dz}{2\pi\I} \!\!\! \oint\limits_{\Gamma_{-1/2,\alpha}} \!\!\! \frac{dw}{2\pi\I}\frac{\Phi(x,z)}{\Phi(y,w)}\left[\frac{\tfrac12-z}{\tfrac12-w}\right]^n\frac{z-\alpha}{w-\alpha}\frac{z+w}{2z(z-w)} \\
               =& -\overline{\mathsf{K}}_{21} (y,x),\\
   \overline{\mathsf{K}}_{22}(x,y) =&\ \e(x,y)+\oint \frac{dz}{2\pi\I}\oint\frac{dw}{2\pi\I} \frac{\Phi(x,z)}{\Phi(y,w)} \frac{1}{\left[(\tfrac12+z)(\tfrac12-w)\right]^n}\frac{1}{z-w}\left(\frac{1}{z+\alpha}+\frac{1}{w-\alpha}\right)
  \end{split}
 \end{equation}
 where the integration contours for $\overline{\mathsf{K}}_{22}$ are $\Gamma_{1/2,-\alpha}\times\Gamma_{-1/2}$ for the term with $1/(z+\alpha)$ and $\Gamma_{1/2}\times\Gamma_{-1/2,\alpha}$ for the term with $1/(w-\alpha)$.
 \end{lem}
\begin{proof}
Analyticity for $\overline{K}_{11}$ (respectively $\overline{K}_{12}$) is obvious since we can take the integration contour for $z$ as close to $1/2$ as desired and the contour for $w$ to include $-1/2$ (respectively $-1/2,-\alpha,-\beta$) without crossing $z$. For $\overline{K}_{22}$, we can decompose the kernel using the identities
\begin{equation}
\frac{(z+\beta)(w-\beta)}{(z-\beta)(w+\beta)(z-w)}=\frac{1}{z-w}+\frac{2\beta}{(w+\beta)(z-\beta)},\quad \frac{z+w}{(z-\alpha)(w+\alpha)}=\frac{1}{z-\alpha}+\frac{1}{w+\alpha}.
\end{equation}
We get that the double integral part $\overline{K}_{22} - \tilde \varepsilon$ becomes the sum of
\begin{equation}\label{eq:term_1}
\oint \frac{dz}{2\pi\I}\oint\frac{dw}{2\pi\I} \frac{\Phi(x,z)}{\Phi(y,w)} \frac{1}{\left[(\tfrac12+z)(\tfrac12-w)\right]^n}\frac{1}{z-w}\left(\frac{1}{z-\alpha}+\frac{1}{w+\alpha}\right)
\end{equation}
and
\begin{equation}\label{eq:term_2}
\oint \frac{dz}{2\pi\I}\oint\frac{dw}{2\pi\I} \frac{\Phi(x,z)}{\Phi(y,w)} \frac{1}{\left[(\tfrac12+z)(\tfrac12-w)\right]^n} \frac{2\beta}{(w+\beta)(z-\beta)}\left(\frac{1}{z-\alpha}+\frac{1}{w+\alpha}\right).
\end{equation}
The integration contour for the term in~\eqref{eq:term_1} with $1/(z-\alpha)$ is $\Gamma_{1/2,\alpha}\times\Gamma_{-1/2}$, while the one for the term with $1/(w+\alpha)$ is $\Gamma_{1/2}\times\Gamma_{-1/2,-\alpha}$. These can be chosen non-intersecting for all $\alpha$ in any subset of $[-1/2+\epsilon,1/2-\epsilon]$ for $0 < \epsilon \ll 1$. The contours for the term in~\eqref{eq:term_2} with $1/(z-\alpha)$ are $\Gamma_{1/2,\alpha,\beta}\times\Gamma_{-1/2,-\beta}$, while the ones with $1/(w+\alpha)$ are $\Gamma_{1/2,\beta}\times\Gamma_{-1/2,-\alpha,-\beta}$. Notice that since the term $1/(z-w)$ is absent, the contours can cross without problems. Thus this term is also clearly analytic.

Comparing~\eqref{eq:336} with~\eqref{eq:h_bar_def}, one notices the change of $\tilde\varepsilon$ into $\varepsilon$, which corresponds to replacing $\varepsilon_2$ with $\varepsilon_0$. Let us start with~\eqref{eq:h_bar_def}. The contributions of the poles at $(z,w)=(1/2,-\alpha)$ and at $(z,w)=(\alpha,-1/2)$ vanish as $\beta\to -\alpha$. The contributions from $(z,w)=(1/2,-1/2)$, $(z,w)=(1/2,-\beta)$ and $(z,w)=(\beta,-1/2)$ give, in the limit $\beta\to-\alpha$, the double integral in~\eqref{eq:336}. Finally, the contributions from $(z,w)=(\alpha,-\beta)$ and $(z,w)=(\beta,-\alpha)$ become, in the $\beta\to-\alpha$ limit, as follows:
\begin{equation}
\frac{e^{-\alpha(x-y)}-e^{\alpha(x-y)}}{(\tfrac14-\alpha^2)^n}= \sgn(x-y) \frac{e^{-\alpha|x-y|}-e^{\alpha|x-y|}}{(\tfrac14-\alpha^2)^n}.
\end{equation}
Summing this term with $\varepsilon_2$ gives $\varepsilon_0$.
\end{proof}

\begin{prop}\label{prop:cvgOfFredPf}
$\pf (J- \overline{K})$ is analytic in $\alpha, \beta \in (-1/2, 1/2)$ and with a well-defined limit
\begin{equation}
\lim_{\beta\to -\alpha}  \pf (J- \overline{K}) = \pf (J - \overline{\mathsf{K}}).
\end{equation}
\end{prop}
\begin{proof}
Fix a $0<\epsilon\ll 1$ and consider $\alpha,\beta\in [-1/2+\epsilon,1/2-\epsilon]$. Take $\mu=1/2-3\epsilon/4$. Then, for some constant $C$ independent of $x,y$, we have the bounds
\begin{equation}\label{eq:246}
\begin{aligned}
|\overline{K}_{11}(x,y)|& \leq C e^{-(1/2-\epsilon/2)x} e^{-(1/2-\epsilon/2)y}=C e^{-\mu (x+y)}e^{-\epsilon (x+y)/4},\\
|\overline{K}_{12}(x,y)|&\leq C e^{-(1/2-\epsilon/2)x} e^{(1/2-\epsilon)y}=C e^{-\mu (x-y)}  e^{-\epsilon (x+y)/4},\\
| \overline{K}_{21}(x,y)|&\leq C e^{(1/2-\epsilon)x} e^{-(1/2-\epsilon/2)y}=C e^{\mu (x-y)}e^{-\epsilon (x+y)/4},\\
|\overline{K}_{22}(x,y)|&\leq C e^{(1/2-\epsilon)x} e^{(1/2-\epsilon)y}=C e^{\mu (x+y)} e^{-\epsilon (x+y)/4}.
\end{aligned}
\end{equation}
This is achieved as follows: for $\overline{K}_{11}$, choose the contours as $|z-1/2|=\epsilon/2$ and $|w+1/2|=\epsilon/2$; for $\overline{K}_{12}$, take $|z-1/2|=\epsilon/2$ and notice the poles at $w=-\alpha,-\beta$ give the leading asymptotic behavior in $y$, namely $e^{-\min\{\alpha,\beta\} y}$. This is controlled by the $e^{-\mu y}$ from the conjugation. The situation is similar for $\overline{K}_{22}$. In this case the leading behavior is given by the residues at $\pm\alpha$ and $\pm\beta$.

Then
\begin{equation}
\begin{aligned}
\pf (J- \overline{K}) &= \sum_{n\geq 0} \frac{(-1)^n}{n!} \int_s^\infty dx_1\cdots \int_s^\infty dx_n \pf[\overline{K}(x_i,x_j)]_{1\leq i<j\leq 2n}\\
&=\sum_{n\geq 0} \frac{(-1)^n}{n!} \int_s^\infty dx_1\cdots \int_s^\infty dx_n \left(\det[\overline{G}(x_i,x_j)]_{1\leq i,j\leq 2n}\right)^{1/2}.
\end{aligned}
\end{equation}
Using the standard Hadamard bound on the $2n\times 2n$ determinant in the sum together with the estimates from~\eqref{eq:246}, we have that the Fredholm expansion of $\pf(J-\overline{K})$ is absolutely convergent. Furthermore, each entry of the series is analytic in the claimed domain, and thus so is $\pf(J-\overline{K})$.
\end{proof}

\subsubsection{Analyticity of the term $\frac{1}{\alpha+\beta} - \braket{Y_2}{X_2}$} \label{sec:analyticity_1}

The analyticity of the term $\frac{1}{\alpha+\beta} - \braket{Y_2}{X_2}$ is relatively easy.
 \begin{lem} \label{lem:inner1}
  The term $\frac{1}{\alpha+\beta} - \braket{Y_2}{X_2}$ is analytic for $\alpha, \beta \in (-1/2, 1/2)$, with
  \begin{equation}
    \begin{split}
     \lim_{\beta\to-\alpha} \left( \frac{1}{\alpha+\beta}  -  \braket{Y_2}{X_2} \right) = -\oint\limits_{\Gamma_{1/2,\alpha}}\frac{dz}{2\pi\I} \frac{\Phi(s,z)}{\Phi(s,\alpha)}\frac{(\tfrac12+\alpha)^n}{(\tfrac12+z)^n}\frac{1}{(z-\alpha)^2}.
    \end{split}
   \end{equation}
 \end{lem}
\begin{proof}
We have
\begin{equation}
\begin{aligned}
\braket{Y_2}{X_2} = \braket{g_2}{f_{+}^\beta}&= \! \oint\limits_{\Gamma_{1/2,\alpha}} \! \frac{dz}{2\pi\I} \frac{\Phi(s,z)}{\Phi(s,-\beta)}\frac{(\tfrac12-\beta)^n}{(\tfrac12+z)^n}\frac{1}{(z-\alpha)(z+\beta)}\\
&= \!\!\! \oint\limits_{\Gamma_{1/2,\alpha,\beta}} \!\!\! \frac{dz}{2\pi\I} \frac{\Phi(s,z)}{\Phi(s,-\beta)}\frac{(\tfrac12-\beta)^n}{(\tfrac12+z)^n}\frac{1}{(z-\alpha)(z+\beta)}+\frac{1}{\alpha+\beta}
\end{aligned}
\end{equation}
where the last term is the residue at $z=-\beta$ chosen so that the first term is analytic. The residue exactly cancels the $1/(\alpha+\beta)$ in $\frac{1}{\alpha+\beta} - \braket{Y_2}{X_2}$. The latter is therefore analytic with limit
\begin{equation}
\lim_{\beta\to -\alpha} \frac{1}{\alpha+\beta} - \braket{Y_2}{X_2} = - \!\!\! \oint\limits_{\Gamma_{1/2,\alpha}} \!\!\! \frac{dz}{2\pi\I} \frac{\Phi(s,z)}{\Phi(s,\alpha)}\frac{(\tfrac12+\alpha)^n}{(\tfrac12+z)^n}\frac{1}{(z-\alpha)^2}.
\end{equation}
\end{proof}

\subsubsection{Analyticity of the term $\braket{ Y_2 } { (\Id-\overline G)^{-1}\overline G X_2 }$} \label{sec:analyticity_2}

Finding an analytic decomposition of $\braket{Y_2}{(\Id-\overline G)^{-1}\overline G X_2}$ turns out to be more intricate than in the full-space stationary case. Let us explain first where the problems are. We have
\begin{equation}
\overline G \ket{X_2} = \ket{\begin{array}{c}-\overline K_{22} f_{+}^\beta\\ \overline K_{12} f_{+}^\beta\end{array}}.
\end{equation}
The issues are the following:
\begin{itemize}
\item[(a)] The pole at $w=-\beta$ of $\overline K_{22}$ leads to a term of the form $\ketbra{a}{f_{-}^\beta}$ for some explicit function $a$, and similarly for $\overline K_{12}$. When these terms are multiplied by $\ket{f_{+}^\beta}$, they give terms proportional to $\int_s^\infty dy e^{-2\beta y}<\infty$ iff $\beta>0$ whereas the model is defined for any $\alpha, \beta$ with $\alpha+\beta>0$.
\item[(b)] The pole at $w=-\alpha$ of $\overline K_{22}$ leads to terms of the form $\ketbra{a}{f_{-}^\alpha}$, and similarly for $\overline K_{12}$. When multiplied with $\ket{f_{+}^\beta}$, and taking into account the prefactors, one gets a term proportional to $1/(\beta-\alpha)$. This is well-defined in the $\beta \to -\alpha$ limit, except when $\alpha=0$. Also, the single terms in $1/(\beta-\alpha)$ are not analytic at $\alpha=\beta$. This would not be a serious problem if we did not want to consider also the $\alpha=\beta=0$ case, which we of course do.
\end{itemize}

Thus what we have to prove is that the terms in (a) give a zero contribution, within the product $\braket{Y_2}{(\Id-\overline G)^{-1}\overline G X_2}$ for any $\beta>0$; we also need to rewrite (b) such that we do not have divergent terms for $\beta=\alpha$. These issues did not occur in the full-space stationary problem, but can be put under control using the $2\times 2$ structure of our kernels.

The idea to overcome this issue is the following. We decompose
\begin{equation}\label{eq:decomposition}
\overline{G}=\widehat{G}+O,\quad \textrm{with}\quad O = \ketbra{\begin{array}{c} g_2 \\ g_1 \end{array}}{a \quad b}
\end{equation}
for some functions $a,b$ to be written down in the sequel. The following lemma tells us that the matrix kernel $O$ is irrelevant in $\braket{Y_2}{(\Id-\overline G)^{-1}\overline G X_2}$.

\begin{lem} \label{lem:GbarToGhat}
  We have:
  \begin{equation}
  \braket{Y_2}{(\Id - \overline{G})^{-1} \overline{G} X_2} = \braket{Y_2}{(\Id - \overline{G})^{-1} \widehat{G} X_2}.
  \end{equation}
\end{lem}
\begin{proof}
First of all, notice that
\begin{equation}
(\Id - \overline{G})^{-1} \overline{G}-(\Id - \overline{G})^{-1} \widehat{G} = (\Id - \overline{G})^{-1}  O.
\end{equation}
Therefore
\begin{equation} \label{eq:diff_inner}
\braket{Y_2} {(\Id - \overline{G})^{-1} \overline{G} X_2} - \braket{Y_2}{(\Id - \overline{G})^{-1} \widehat{G} X_2} =  \textrm{const} \braket{-g_1 \quad g_2} {(\Id - \overline{G})^{-1}  \begin{pmatrix} g_2 \\ g_1 \end{pmatrix}}
\end{equation}
with ${\rm const}=\braket{a \quad b} { X_2}$. Multiplying the $2\times 2$ matrices, we have that the scalar product without the constant is given by
\begin{equation}
-\braket{g_1} {(\Id - \overline{G})^{-1}_{11} g_2} + \braket{g_2}{(\Id - \overline{G})^{-1}_{22} g_1} - \braket{g_1} {(\Id - \overline{G})^{-1}_{12} g_1} + \braket{g_2} {(\Id - \overline{G})^{-1}_{21} g_2}.
\end{equation}
The property $\overline{K}_{12}(x,y)=-\overline{K}_{21}(y,x)$ translates into $\overline{G}_{11}(x,y)=\overline{G}_{22}(y,x)$ and so into \mbox{$(\Id - \overline{G})^{-1}_{11}(x,y) = (\Id - \overline{G})^{-1}_{22}(y,x)$}, see Proposition~\ref{prop:kernel_symm} for details, implying the first two terms cancels each other. The anti-symmetry of $\overline{K}_{11}$ and $\overline{K}_{22}$ implies the anti-symmetry of $\overline{G}_{21}$ and $\overline{G}_{12}$ which in turn implies the same for the respective $(\Id-\overline{G})^{-1}$ entries. Thus the last two terms are each equal to zero and this finishes the proof.
\end{proof}

We now state the announced further decomposition of $\overline{K}$.

\begin{prop} \label{prop:Seconddecomposition} Let $\alpha \in (-1/2,1/2)$, $\beta>0$. Then the kernel $\overline{K}$ splits as
 \begin{equation}
  \overline{K}=\widetilde{K} + \begin{pmatrix} 0 & 0 \\ 0  & \tilde \varepsilon \end{pmatrix} +\widetilde O+\widetilde P
 \end{equation}
 where
 \begin{equation}
 \begin{aligned}
 \widetilde{K}_{11} &=\overline{K}_{11}, \quad \widetilde{K}_{21}=\overline{K}_{21},\\
 \widetilde{K}_{12}(x,y) &=-\oint\limits_{\Gamma_{1/2}} \frac{dz}{2\pi\I} \oint\limits_{\Gamma_{-1/2}} \frac{dw}{2\pi\I}\frac{\Phi(x,z)}{\Phi(y,w)} \left[ \frac{\tfrac12-z}{\tfrac12-w} \right]^n\frac{z+\alpha}{w+\alpha}\frac{z+\beta}{z-\beta}\frac{w-\beta}{w+\beta}\frac{z+w}{2z(z-w)},\\
 \widetilde{K}_{22}(x,y) &=\oint\limits_{\Gamma_{1/2,\alpha,\beta}} \frac{dz}{2\pi\I} \oint\limits_{\Gamma_{-1/2}} \frac{dw}{2\pi\I} \frac{\Phi(x,z)}{\Phi(y,w)} \frac{1}{\left[(\tfrac12+z)(\tfrac12-w)\right]^n}\frac{1}{(z-\alpha)(w+\alpha)} \frac{z+\beta}{z-\beta}\frac{w-\beta}{w+\beta}\frac{z+w}{z-w}
 \end{aligned}
 \end{equation}
and
 \begin{equation}
 \begin{aligned}
\widetilde O &= \ketbra{\begin{array}{c} -g_1 \\ g_2 \end{array}}{0 \quad \tfrac{2\beta}{\beta-\alpha}f_{-}^\beta-\tfrac{\alpha+\beta}{\beta-\alpha}f_{-}^\alpha},\\
\widetilde P &= \ketbra{\begin{array}{c} (\alpha+\beta)g_3 \\ -(\alpha+\beta) g_4-f_{-}^{-\alpha} \end{array}}{0 \quad f_{-}^\alpha}.
 \end{aligned}
 \end{equation}
 \end{prop}
 \begin{proof}
We first have to compute the $(12)$ and $(22)$ components of $\widetilde O+\widetilde P$ and then divide up accordingly. We have
\begin{equation}
\widetilde O_{12} + \widetilde P_{12}=\oint\limits_{\Gamma_{1/2}}\oint\limits_{\Gamma_{-\alpha,-\beta}}\cdots.
\end{equation}
The residue computations at $w=-\alpha$ and $w=-\beta$ lead to
\begin{equation} \label{eq:g_eq_1}
\begin{aligned}
\widetilde O_{12} + \widetilde P_{12}&= -\frac{2\beta}{\beta-\alpha} \ketbra{g_1}{f_{-}^\beta}+\frac{\alpha+\beta}{\beta-\alpha}\ketbra{g_5}{f_{-}^\alpha}\\
&=-\frac{2\beta}{\beta-\alpha} \ketbra{g_1}{f_{-}^\beta}+\frac{\alpha+\beta}{\beta-\alpha}\ketbra{g_1}{f_{-}^\alpha}+(\alpha+\beta)\ketbra{g_3}{f_{-}^\alpha}
\end{aligned}
\end{equation}
where in the second equality we used Lemma~\ref{lem:identities}. Similarly,
\begin{equation}
\widetilde O_{22} + \widetilde P_{22}=\oint\limits_{\Gamma_{1/2,\alpha}}\oint\limits_{\Gamma_{-\beta}}\cdots + \oint\limits_{\Gamma_{1/2,\beta}}\oint\limits_{\Gamma_{-\alpha}}\cdots.
\end{equation}
Computing the residues at $w=-\alpha$ and $w=-\beta$ we get
\begin{equation} \label{eq:g_eq_2}
\begin{aligned}
\widetilde O_{22} + \widetilde P_{22}&= \frac{2\beta}{\beta-\alpha} \ketbra{g_2}{f_{-}^\beta}-\frac{2\beta}{\beta-\alpha}\ketbra{f_{-}^\beta}{f_{-}^\alpha}- \frac{\alpha+\beta}{\beta-\alpha}\ketbra{g_6}{f_{-}^\alpha}\\
&=\frac{2\beta}{\beta-\alpha} \ketbra{g_2}{f_{-}^\beta}
-\frac{\alpha+\beta}{\beta-\alpha}\ketbra{g_2}{f_{-}^\alpha}-(\alpha+\beta)\ketbra{g_4}{f_{-}^\alpha}-\ketbra{f_{-}^{-\alpha}}{f_{-}^\alpha}.
\end{aligned}
\end{equation}
Recombining the terms leads to the claimed result, provided we prove the two equalities used in equations~\eqref{eq:g_eq_1} and~\eqref{eq:g_eq_2}. We do this in the next lemma.
\end{proof}

Now we prove the two identities used in the proof of Proposition~\ref{prop:Seconddecomposition}.
\begin{lem} \label{lem:identities}
We have the following identities:
\begin{equation}
\begin{split}
g_5(x)&=g_1(x)+(\beta-\alpha) g_3(x), \\
g_6(x)+\frac{2\beta}{\alpha+\beta} f_{-}^\beta(x) &= g_2(x)+ (\beta-\alpha)\left(g_4(x) + \frac{1}{\alpha+\beta}f_{-}^{-\alpha}(x)\right).
\end{split}
\end{equation}
\end{lem}
\begin{proof}
The first identity follows directly from the relation $\frac{(z-\alpha)(z+\beta)}{z-\beta}-(z+\alpha)=\frac{\beta-\alpha}{z-\beta}$. To prove the second, first rewrite
\begin{equation}\label{eq:g4}
g_6(x)+\frac{2\beta}{\alpha+\beta} f_{-}^\beta(x) =\oint\limits_{\Gamma_{1/2,\beta}}\frac{dz}{2\pi\I} \frac{\Phi(x,z)}{(\tfrac12+z)^n}\frac{z+\beta}{(z+\alpha)(z-\beta)}
\end{equation}
and recall
\begin{equation}\label{eq:g2}
 g_2(x) = \oint\limits_{\Gamma_{1/2,\alpha}}\frac{dz}{2\pi\I} \frac{\Phi(x,z)}{(\tfrac12+z)^n}\frac{1}{z-\alpha}.
\end{equation}
Taking the same contours (i.e.\,including $1/2,\alpha,\beta$ inside both) and then computing the difference~\eqref{eq:g4}$-$\eqref{eq:g2} leads to $(\beta-\alpha) g_4(x)$ minus the pole coming from $z=-\alpha$ in $g_4(x)$. The latter is $-\frac{1}{\alpha+\beta}f_{-}^{-\alpha}(x)$ and this finishes the proof.
\end{proof}

Coming back to the decomposition~\eqref{eq:decomposition}, namely
\begin{equation}
\overline G =\widehat G + O
\end{equation}
with $O=J^{-1} \widetilde O$, we notice the latter has exactly the form to apply Lemma~\ref{lem:GbarToGhat}. We also explicitly have
\begin{equation} \label{eq:g_hat_decomp}
\widehat G=\begin{pmatrix}
             -\widetilde K_{21} & -\widetilde K_{22} - \tilde \e \\ \widetilde K_{11} & \widetilde K_{12}
         \end{pmatrix}+ \begin{pmatrix}
             0 & \ketbra{(\alpha+\beta) g_4 + f_{-}^{-\alpha}} { f_{-}^\alpha} \\
             0 & (\alpha+\beta) \ketbra{g_3} {f_{-}^\alpha}
         \end{pmatrix}.
\end{equation}

What remains to be done is to show that $\braket{Y_2}{(\Id - \overline{G})^{-1} \widehat{G} X_2}$ is analytic for \mbox{$\alpha,\beta\in (-1/2,1/2)$} and determine its $\beta\to-\alpha$ limit. This will be accomplished in Proposition~\ref{prop:AnalyticityBigScalarProduct}, itself following from the following three lemmas.

\begin{lem} \label{lem:y2}
The vector $Y_2(x)= (-g_1(x) \quad g_2(x))$ is independent of $\beta$. Furthermore, for any $\alpha\in [-1/2+\epsilon,1/2-\epsilon]$,
\begin{equation}
    |g_1(x)| \leq C e^{-x (1/2-\epsilon/2)}, \quad
    |g_2(x)| \leq C e^{(1/2-\epsilon) x}
\end{equation}
for some constant $C$ uniform in $x$.
\end{lem}

\begin{proof}
The bound on $g_1$ is simply obtained by taking the integration contour $|z-1/2|=\epsilon/2$, while for $g_2$ the leading asymptotics comes from the pole at $z=\alpha$.
\end{proof}

\begin{lem} \label{lem:g_x2}
$\widehat G X_2$ is analytic in $\alpha, \beta \in (-1/2, 1/2)$ and, for any $\alpha,\beta\in [-1/2+\epsilon,1/2-\epsilon]$, we have the following bounds:
\begin{equation}
\left|(\widehat G X_2)_1(y) \right| \leq C e^{y(1/2-\epsilon)}, \quad \left| (\widehat G X_2)_2(y) \right| \leq C e^{-y(1/2-\epsilon/2)}
\end{equation}
for some constant $C$ independent of $y$. Moreover we have
 \begin{equation}
  \lim_{\beta \to -\alpha} \widehat{G} \ket{X_2} = \ket{\begin{array}{c} -h_1 \\ h_2 \end{array}}
 \end{equation}
 where
 \begin{equation} \label{eq:h_def}
  \begin{split}
  h_1 &= \widetilde{\mathsf{K}}_{22} f_{+}^{-\alpha} + \varepsilon_1 f_{+}^{-\alpha} - \mathsf{g}_4 - j^\alpha(s, \cdot), \\
  h_2 &= \widetilde{\mathsf{K}}_{12} f_{+}^{-\alpha} + \mathsf{g}_3
  \end{split}
 \end{equation}
 with
 \begin{equation} \label{eq:j_def}
  j^{\alpha}(s, y) =  f_{-}^{-\alpha} (s) \left[\frac{\sinh(\alpha(y-s))}{\alpha} +(y-s) e^{\alpha (y-s)}  \right].
 \end{equation}
\end{lem}

\begin{proof}
We start with
 \begin{equation}
  \widehat{G} \ket{X_2} = \begin{pmatrix} * & a \\ * & b \end{pmatrix} \ket{ \begin{array}{c} 0 \\ f_{+}^\beta \end{array} } = \ket{ \begin{array}{c} a f_{+}^\beta \\ b f_{+}^\beta \end{array}  }
 \end{equation}
where the kernels $a, b$ are read from the decomposition~\eqref{eq:g_hat_decomp}, namely
 \begin{equation}
  a = -\widetilde K_{22} - \varepsilon_1 - \varepsilon_2 + (\alpha+\beta) \ketbra{g_4}{f_{-}^\alpha} + \ketbra{f_{-}^{-\alpha}}{f_{-}^\alpha}, \quad b =  \widetilde K_{12} + (\alpha+\beta) \ketbra{g_3} {f_{-}^\alpha}.
 \end{equation}

The terms which in the $y$ variable decay like $e^{-y(1-\e)/2}$, i.e.\,the ones for which in the integral representation we integrate only around the pole at $w=-1/2$, are clearly analytic when multiplied by $f_{+}^\beta$ and the limits are straightforward, namely
\begin{equation}
\begin{aligned}
\lim_{\beta\to -\alpha}\widetilde{K}_{12}f_{+}^\beta & = \widetilde{\mathsf{K}}_{12} f_{+}^{-\alpha},\\
\lim_{\beta\to -\alpha}-(\widetilde{K}_{22}f_{+}^\beta + \varepsilon_1 f_{+}^\beta) & =  -(\widetilde{\mathsf{K}}_{22} f_{+}^{-\alpha}+\varepsilon_1 f_{+}^{-\alpha}).
\end{aligned}
\end{equation}

Next we compute $\braket{f_{-}^\alpha}{f_{+}^\beta}$ with the result being
\begin{equation} \label{eq:inner_alpha_beta}
\begin{aligned}
\braket{f_{-}^\alpha}{f_{+}^\beta} &= \phi(\alpha) \phi(\beta) \left[ \frac{\tfrac12-\beta}{\tfrac12+\alpha} \right]^n \int_s^\infty e^{-(\alpha + \beta)x} dx\\
&= \phi(\alpha) \phi(\beta) \left[ \frac{\tfrac12-\beta}{\tfrac12+\alpha} \right]^n \frac{e^{-(\alpha + \beta) s}} {\alpha + \beta} = \frac{f_{-}^\alpha(s) f_{+}^\beta(s)}{\alpha+\beta}.
\end{aligned}
\end{equation}
Since $f_{-}^\alpha,f_{+}^\beta,g_3,g_4$ are clearly analytic, see their representations in~(\ref{eq:fg_def}), then also the two terms involving $g_3$ and $g_4$ are analytic as the $\alpha+\beta$ prefactor cancels with the one in~(\ref{eq:inner_alpha_beta}). Their limits are given by
\begin{equation}
  \lim_{\beta\to -\alpha}(\alpha+\beta) g_3 \ketbra{f_{-}^\alpha}{f_{+}^\beta} =\mathsf{g}_3, \quad
  \lim_{\beta\to -\alpha}(\alpha+\beta) g_4 \ketbra{f_{-}^\alpha}{f_{+}^\beta} =\mathsf{g}_4.
\end{equation}

Finally, it remains to analyze the term $-\varepsilon_2+\ketbra{f_{-}^{-\alpha}}{f_{-}^\alpha}$ applied to $f_{+}^\beta$. We have
 \begin{equation} \label{eq:eps_fh}
 \begin{split}
   \braket{f_{-}^\alpha}{f_{+}^\beta} f_{-}^{-\alpha}(y) - \varepsilon_2 f_{+}^\beta(y) &= \frac{\phi(\beta) (\tfrac12-\beta)^n}{(\tfrac14-\alpha^2)^n} \left[ \frac{e^{-(\alpha + \beta)s + \alpha y}}{\alpha+\beta} + \int_s^\infty\!\!\!\sgn(y-u) e^{-\alpha|y-u|-\beta u} du \right] \\
  &= \frac{\phi(\beta) (\tfrac12-\beta)^n}{(\tfrac14-\alpha^2)^n} \left[ \frac{e^{-(\alpha + \beta)s + \alpha y}}{\alpha+\beta} + \frac{e^{- \beta y}-e^{(\alpha - \beta) s - \alpha y}} {\alpha-\beta} - \frac{e^{- \beta y}}{\alpha+\beta} \right] \\
  &= \frac{\phi(\beta) (\tfrac12-\beta)^n e^{-\beta s}}{(\tfrac14-\alpha^2)^n} \left[ \frac{e^{\beta (s-y)}-e^{\alpha(s-y)}}{\alpha - \beta} + \frac{ e^{\alpha(y-s)} - e^{-\beta (y-s)}} {\alpha+\beta}\right]
 \end{split}
 \end{equation}
 where in the second line inside the brackets, the second and third terms come from explicitly integrating $\int_s^y du \cdots$ and $\int_y^\infty du \cdots$ respectively. The $\beta\to -\alpha$ of this last term is
 \begin{equation}
 \lim_{\beta\to -\alpha} \braket{f_{-}^\alpha}{f_{+}^\beta} f_{-}^{-\alpha}(y) - \varepsilon_2 f_{+}^\beta(y) = \frac{\phi(-\alpha)e^{\alpha s}}{(\tfrac12-\alpha)^n} \left[\frac{\sinh(\alpha(y-s))}{\alpha} +(y-s) e^{\alpha (y-s)}  \right].
\end{equation}

It remains to discuss the decay properties. The bound in the first component follows directly from the fact that in the representation of $\widetilde{K}_{12}$ and $g_3$ we can take the contour as $|z-1/2|=1/2-\epsilon/2$. This decay is also correct for the term involving $\varepsilon_1$. Furthermore, the asymptotic behavior of $\widetilde{K}_{22}$ is coming from the poles at $z=\alpha,\beta$, and thus is $e^{-\min\{\alpha,\beta\} y}$. Similarly for $g_4$, the behavior is $e^{-\min\{\alpha,\pm \beta\} y}$. Finally, the behavior of~\eqref{eq:eps_fh} is clear from its final form. Thus, by choosing $\alpha,\beta\in [-1/2+\epsilon,1/2-\epsilon]$, the claimed bounds holds.
\end{proof}

\begin{lem} \label{lem:g_hat}
The kernel $\overline{G}$ is analytic for $\alpha, \beta \in (-1/2, 1/2)$. Moreover, for any $\alpha, \beta \in [-1/2+\epsilon,1/2-\epsilon]$ we have the following bounds:
\begin{equation}
\begin{aligned}
|\overline{G}_{11}(x,y)| &\leq C e^{(1/2-\epsilon) x} e^{-(1/2-\epsilon/2) y}, \\
|\overline{G}_{12}(x,y)| &\leq C e^{(1/2-\epsilon) x} e^{(1/2-\epsilon) y}, \\
|\overline{G}_{21}(x,y)| &\leq C e^{-(1/2-\epsilon/2) x} e^{-(1/2-\epsilon/2) y}, \\
|\overline{G}_{22}(x,y)| &\leq C e^{-(1/2-\epsilon/2) x} e^{(1/2-\epsilon) y}
\end{aligned}
\end{equation}
for some constant $C$ independent of $x,y$.
\end{lem}

\begin{proof}
It is a rewriting of the results of Lemma~\ref{lem:Kbar} and Proposition~\ref{prop:cvgOfFredPf}.
\end{proof}

We conclude with the analyticity of the sought inner product.

\begin{prop}\label{prop:AnalyticityBigScalarProduct}
The term $\braket{Y_2}{(\Id - \overline{G})^{-1} \widehat{G} X_2}$ is analytic for $\alpha, \beta \in (-1/2, 1/2)$. Its $\beta \to -\alpha$ limit is given by
\begin{equation}
\lim_{\beta\to -\alpha} \braket{Y_2}{(\Id - \overline{G})^{-1} \widehat{G} X_2} =
\braket{-g_1 \quad g_2} { \begin{pmatrix} \mathsf{R}_{11} & \mathsf{R}_{12} \\ \mathsf{R}_{21} & \mathsf{R}_{22} \end{pmatrix} \begin{pmatrix} -h_1 \\ h_2 \end{pmatrix} }
\end{equation}
where for brevity we denoted $\mathsf{R} = (\Id - \overline{\mathsf{G}})^{-1}$ and where $h_1, h_2$ are as in~\eqref{eq:h_def}.
\end{prop}

\begin{proof}
Due to the bounds of Lemma~\ref{lem:y2},~\ref{lem:g_x2}, and~\ref{lem:g_hat}, when multiplying the different terms, in each integral on $(s,\infty)$ we have integrands bounded for instance by $e^{-\epsilon x/2}$. Thus the product is well-defined. Analyticity follows from the analyticity of the different entries of the scalar product.
\end{proof}

\begin{rem}
  \label{rem:Inverse}
The formula we obtained might not look very practical to get numerical results due to the $(\Id-\overline{G})^{-1}=(\Id-J^{-1}\overline{K})^{-1}$ term. However, we can always write \mbox{$\pf(\Id-\overline{K})\braket{Y_2}{(\Id - J^{-1}\overline{K})^{-1} \widehat{G} X_2}$} as a difference of two Fredholm Pfaffians, see Lemma~\ref{lem:3.19}. This kind of property has been noticed already by Imamura--Sasamoto~\cite{IS12} in the context of the stationary KPZ equation. Thus one does not strictly speaking ever need to verify that the inverse is well-defined; the formulation with the inverse can be thought merely as compact notation for~\eqref{eq:383} below.
\end{rem}

\begin{lem}\label{lem:3.19}
Let $K$ be a $2\times 2$ anti-symmetric kernel and $J(x,y)=\delta_{x,y} \left( \begin{smallmatrix} 0 & 1 \\ -1 & 0 \end{smallmatrix} \right)$. Let $a,b,c,d$ be functions such that the scalar products in the following formulas are well-defined. Then,
\begin{multline}\label{eq:383}
\pf\left(J-K\right) \braket {c \quad d} { (\Id-J^{-1}K)^{-1} \begin{pmatrix} a \\ b \end{pmatrix} }  \\
=\pf\left(J-K\right)-\pf\bigg(J-K-\ketbra{\begin{array}{c} b \\ -a \end{array}}{c \quad d}-\ketbra{\begin{array}{c} c \\ d \end{array}}{-b \quad a}\bigg).
\end{multline}
\end{lem}

\begin{proof}
The proof consists of similar computations to those leading to~\eqref{eq:F_det}.
\end{proof}

\begin{proof}[Proof of Theorem~\ref{thm:main_thm_fin}]
The shift argument, Lemma~\ref{lem:shift}, and Theorem~\ref{thm:exp_corr} together give a formula for the distribution for $\alpha,\beta$ with $\beta>0$. The analytic continuation and the limits provided in Proposition~\ref{prop:cvgOfFredPf}, Lemma~\ref{lem:inner1}, and Proposition~\ref{prop:AnalyticityBigScalarProduct} imply the claimed result.
\end{proof}

\section{Large time asymptotics: proof of Theorem~\ref{thm:main_thm_asymp}} \label{sec:proofAsymptotic}

In this section we prove our main asymptotic result. Let us recall the scaling \eqref{eq:ScalingS}, namely
\begin{equation}
s=4 N -2 u 2^{5/3}N^{2/3} + S\, 2^{4/3}N^{1/3}.
\end{equation}
Accordingly, in the functions and/or kernels, we need to scale $x,y$ in the same way, i.e.
\begin{equation}
  (x,y)=4 N -2 u 2^{5/3}N^{2/3} + (X,Y) 2^{4/3} N^{1/3}.
\end{equation}
Also, in the integrals we will consider the change of variables
\begin{equation}
 z=\zeta/(2^{4/3}N^{1/3}), \quad w=\omega/(2^{4/3} N^{1/3}).
\end{equation}
Furthermore, since the Fredholm expansion
\begin{equation}
\pf(J - \overline{\mathsf{K}})_{L^2(s,\infty) \times L^2(s,\infty)}=\sum_{m\geq 0}\frac{(-1)^m}{m!} \int_{s}^\infty \dots \int_{s}^\infty \pf_{1 \leq i < j \leq m} \overline{\mathsf{K}}^{(m)} (x_i, x_j) \prod_{i=1}^m d x_i,
\end{equation}
the $m\times m$ Pfaffian has to be multiplied by the volume element $(2^{4/3}N^{1/3})^m$\footnote{Here $\overline{\mathsf{K}}^{(m)} (x_i, x_j)$ is the $2m \times 2m$ anti-symmetric matrix with $2 \times 2$ block at $(i, j)$ given by $\overline{\mathsf{K}} (x_i, x_j)$ for $1 \leq i, j \leq m$.}. As our Pfaffian kernel has a $2\times 2$ structure, not all the kernel elements have to be multiplied by the same volume element. Indeed, in our case, the rescaled and conjugated kernel elements are as follows:
\begin{equation}
 \begin{aligned}
  {\mathcal{K}}_{11}^{\rm resc}(X,Y) &= (2^{4/3} N^{1/3})^2 2^{2n} {\mathcal K}_{11}(x,y), \\
  {\mathcal{K}}_{12}^{\rm resc}(X,Y) &= 2^{4/3}N^{1/3}{\mathcal K}_{12}(x,y), \\
  {\mathcal{K}}_{22}^{\rm resc}(X,Y) &= 2^{-2n}{\mathcal K}_{22}(x,y)
 \end{aligned}
\end{equation}
where $\mathcal{K}$ can stand for $\overline{\mathsf{K}}$ or for $\widetilde{K}$. We also set ${\cal E}^{\rm resc}_k(X,Y)=2^{-2n}\varepsilon_{k}(x,y)$, $k=0,1$ or empty. Similarly, we rescale the functions
\begin{equation}
 f_{+}^{-\delta,{\rm resc}}(X) = 2^n f_{+}^{-\alpha}(x),\quad
  h^{\rm resc}_1(X)= 2^{-4/3} N^{-1/3}2^{-n} h_1(x), \quad h^{\rm resc}_2(X)=2^{n} h_2(x)
\end{equation}
as well as
\begin{alignat}{2}
  e^{\delta, {\rm resc}}(S) &=2^{-4/3} N^{-1/3} e^{\alpha}(s),\quad &
 j^{\delta, {\rm resc}}(S,X) &=2^{-4/3} N^{-1/3} 2^{-n}j^{\delta}(s,x), \nonumber\\
 g^{\rm resc}_1(X) &= 2^{4/3}N^{1/3}2^{-n} g_1(x), \quad & g^{\rm resc}_2 (X) &=2^{-n} g_2(x) , \\
  \mathsf{g}^{\rm resc}_3(X) &= 2^{4/3}N^{1/3}2^{-n} \mathsf{g}_3(x), \quad & \mathsf{g}^{\rm resc}_4(X) &= 2^{-n} \mathsf{g}_4(x).\nonumber
\end{alignat}

\begin{figure}[t!]
 \centering
 \includegraphics[height=4cm]{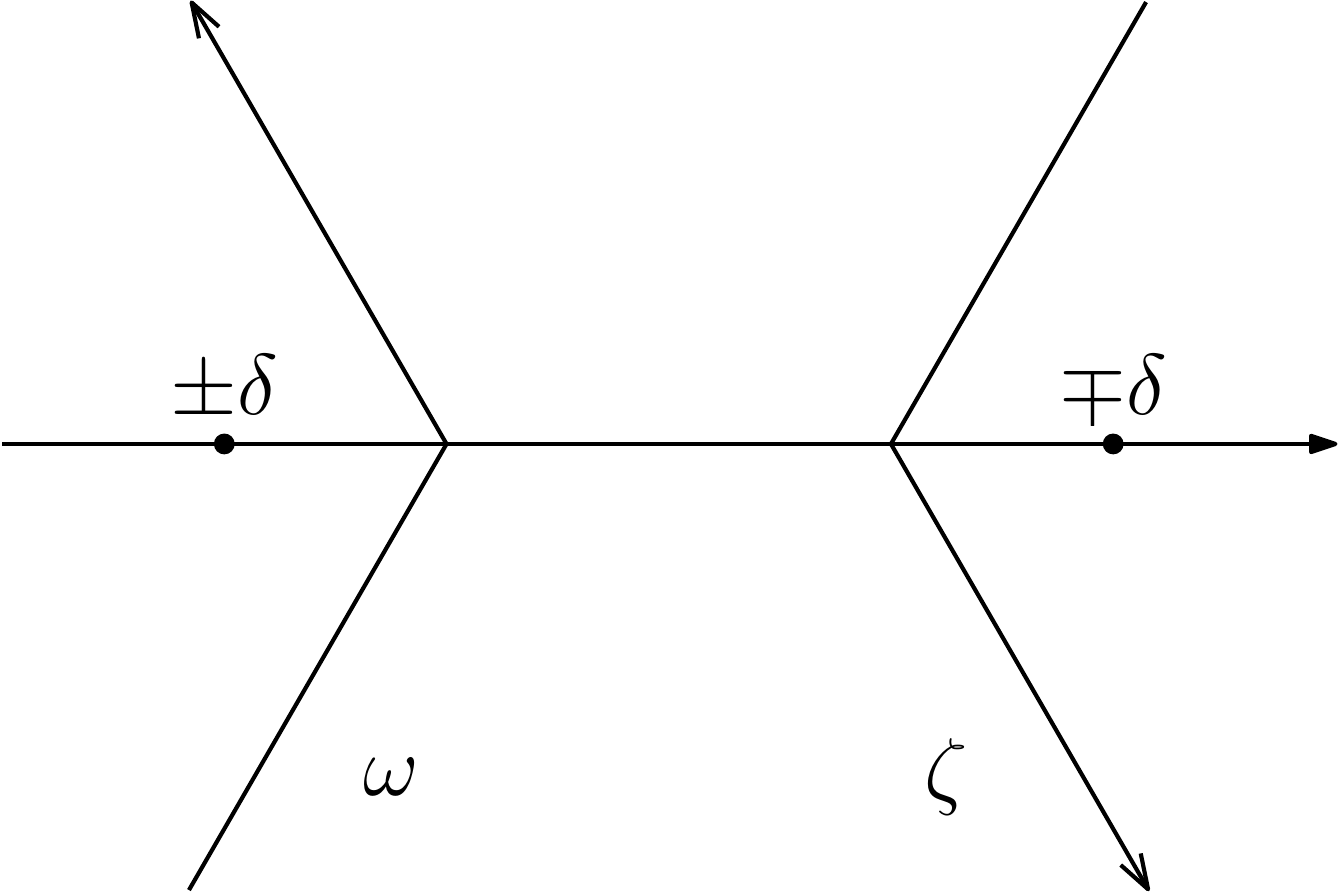}
 \caption{The two Airy integration contours with acute angles of $\pi/3$ with the horizontal axis. Note they have opposite orientations.}
 \label{fig:airy_contours}
\end{figure}

Both the functions and the kernels have a similar structure to the ones of the full-space stationary case, analyzed in great detail in~\cite{BFP09}. The integrals have terms of the form $e^{N f_0(z)+N^{2/3}f_1(z)+N^{1/3}f_2(z)}$, a similar term in $w$ in the case of double integrals, and a finite product of terms independent of $N$. Since the function $f_0$ is the same as the one in~\cite{BFP09}, the steepest descent paths used for the asymptotics and for the uniform bounds are the same. The only minor differences are in the functions $f_1$ and $f_2$ and in the $N$-independent terms, but these do not generate any issues in the asymptotic analysis. For this reason we are not going to repeat all the details of the asymptotic analysis, but rather indicate which Lemmas in~\cite{BFP09} we are using analogously. A detailed description on the general approach in the asymptotic of single integrals having Airy-type behaviors can be found for instance in \cite[Section 6.1]{BF08}. This follows the scheme introduced in our field by Gravner--Tracy--Widom in~\cite{GTW00}.

The limits of the functions entering in the statement of Theorem~\ref{thm:main_thm_fin} are the following.

\begin{lem}\label{lem:AsymptoticsFunctions}
For any given $L>0$, the following limits hold uniformly for $X\in [-L,L]$:
\begin{equation}
\lim_{N\to\infty} f^{-\delta,\rm resc}_+(X)= \mathpzc{f}^{-\delta,u}(X),\quad
\lim_{N\to\infty} e^{\delta,\rm resc}(S)=\mathpzc{e}^{\delta, u} (X),\quad
\lim_{N\to\infty} j^{\delta,\rm resc}(S,X)=  \mathpzc{j}^{\delta, u}(S,X)
\end{equation}
as well as
\begin{equation}
\begin{aligned}
\lim_{N\to\infty} g_1^{\rm resc}(X) =\mathpzc{g}_1^{\delta, u} (X),\quad
&\lim_{N\to\infty} g_2^{\rm resc}(X) = \mathpzc{g}_2^{\delta, u} (X),\\
\lim_{N\to\infty} \mathsf{g}^{\rm resc}_3(X)= \mathpzc{g}_3^{\delta, u} (X) ,\quad
&\lim_{N\to\infty} \mathsf{g}^{\rm resc}_4(X)= \mathpzc{g}_4^{\delta, u} (X).
\end{aligned}
\end{equation}
Furthermore, for any $X\geq -L$, we have the following bounds which hold uniformly in $N$:
\begin{equation}
|f^{-\delta,\rm resc}_+(X)|\leq C e^{\delta X},\quad
|\mathpzc{j}^{\delta, u}(S,X)|\leq C |X|e^{|\delta X|}
\end{equation}
for some constant $C$. For any $\kappa>0$ we have
\begin{equation}
\begin{aligned}
|g_1^{\rm resc}(X)|\leq C e^{-\kappa  X},\quad &|g_2^{\rm resc}(X)|\leq C (e^{-\delta X}+ e^{-\kappa X}),\\
|\mathsf{g}^{\rm resc}_3(X)|\leq C e^{-\kappa X},\quad & |\mathsf{g}^{\rm resc}_4(X)|\leq C (|X|e^{|\delta X|}+e^{-\kappa X}).
\end{aligned}
\end{equation}
\end{lem}

\begin{proof}
Inserting the new variables, we have $f^{-\delta,\rm resc}_+(X)= e^{\delta X} e^{Q_N}$ with $Q_N$ independent of $X$ and with $Q_N\to -\delta^3/3-\delta^2 u$ as $N\to\infty$. The limit of $e^{\delta,\rm resc}(S)$ follows the patterns of \cite[Lemma 4.6]{BFP09}. For $j^{\delta,\rm resc}(S,X)$ we have
\begin{equation}
  j^{\delta,\rm resc}(S,X) = \left[\frac{\sinh{\delta(X-S)}}{\delta}+ (X-S)e^{\delta(X-S)}\right] 2^{-n} f_-^{-\alpha}(s)
\end{equation}
and the last term is analyzed as $f^{-\delta,\rm resc}_+$.

The limits of the $g$ functions and their bounds are obtained as in~\cite[Lemma 4.7]{BFP09}. The terms $C e^{-\kappa X}$ come from the integrals with the contours to the right of the poles $\pm\alpha$ (if present), since the real decay is Airy-like, i.e.\,$e^{-c X^{3/2}}$. The contributions of the poles at $\alpha$ are bounded by $C e^{-\alpha X}$, while the pole of order $2$ in $-\alpha$ is bounded by $C |X| e^{\alpha X}$.
\end{proof}

The limits of the kernels are the following.

\begin{lem}\label{lem:AsymptoticsKernels}
For any given $L>0$, the following limits hold uniformly for $X,Y\in [-L,L]$:
\begin{equation}
\lim_{N\to\infty}\overline{\mathsf{K}}^{\rm resc}_{ij}(X,Y)= \overline{\mathcal{A}}_{ij} (X, Y),\quad i,j\in\{1,2\}.
\end{equation}
Furthermore, for any $X,Y\geq -L$ and $\kappa > 0$, we have the following bounds which hold uniformly in $N$:
{\small \begin{alignat}{2}
|\overline{\mathsf{K}}^{\rm resc}_{11}(X,Y)|&\leq C e^{-\kappa (X+Y)},\quad &|\overline{\mathsf{K}}^{\rm resc}_{12}(X,Y)|&\leq C (e^{-\kappa (X+Y)}+e^{-\kappa X} e^{\delta Y}),\nonumber\\
|\overline{\mathsf{K}}^{\rm resc}_{21}(X,Y)|&\leq C (e^{-\kappa (X+Y)}+e^{\delta X} e^{-\kappa Y}),\quad & |\overline{\mathsf{K}}^{\rm resc}_{22}(X,Y)|&\leq |{\cal E}^{\rm resc}(X,Y)|+C (e^{-\kappa X}e^{\delta Y}+e^{\delta X} e^{-\kappa Y}),\nonumber\\
|{\cal E}^{\rm resc}_0(X,Y)|&\leq C e^{\delta |X-Y|},\quad & |{\cal E}^{\rm resc}_1(X,Y)|&\leq C e^{-(|\delta|+\kappa) |X-Y|}
\end{alignat}}
for some constant $C$.
\end{lem}

\begin{proof}
The asymptotics of the double integrals is as in~\cite[Lemma 4.4]{BFP09} and the uniform bounds as in~\cite[Lemma 4.5]{BFP09}. To get the bounds, we first compute explicitly the poles at $\pm \alpha$ (if they are inside the integration contours), while the rest has an Airy-like decay in both variables, from which we have the terms $e^{-\kappa X}$ and $e^{-\kappa Y}$. For ${\cal E}^{\rm resc}_1(X,Y)$, we take a contour passing on the right of $|\alpha|$ by an amount $\kappa 2^{-4/3} N^{-1/3}$, which can be deformed to become vertical, as the convergence comes from the quadratic term in $Z$.
\end{proof}

Finally, in order to define the limits of $h_1^{\rm resc}$ and $h_2^{\rm resc}$, we need the limits of $\widetilde K^{\rm resc}_{12}$ and $\widetilde K^{\rm resc}_{22}$, which are as follows.

\begin{lem}\label{lem:AsymptoticsKernelsTilde}
For any given $L>0$, the following limits hold uniformly for $X,Y\in [-L,L]$:
\begin{equation}
\lim_{N\to\infty} \widetilde{\mathsf{K}}^{\rm resc}_{12}(X,Y)=\widetilde{\mathcal{A}}_{12} (X, Y),\quad
\lim_{N\to\infty} \widetilde{\mathsf{K}}^{\rm resc}_{22}(X,Y)=\widetilde{\mathcal{A}}_{22} (X, Y).
\end{equation}
Furthermore, for any $X,Y\geq -L$ and $\kappa > 0$, we have the following bounds which hold uniformly in $N$:
\begin{equation}
|\widetilde{\mathsf{K}}^{\rm resc}_{12}(X,Y)|\leq C e^{-\kappa (X+Y)},\quad |\widetilde{\mathsf{K}}^{\rm resc}_{22}(X,Y)|\leq C (e^{-\kappa X}+e^{\delta X}) e^{-\kappa Y}
\end{equation}
for some constant $C$.
\end{lem}

\begin{proof}
The proof is similar to that of Lemma~\ref{lem:AsymptoticsKernels}, with the only difference being that some poles are not present anymore.
\end{proof}

\begin{cor}\label{cor:AsymptoticsFredholmPfaffian}
For any given $S\in\R$, we have
\begin{equation}
\lim_{N\to\infty} \pf(J-\overline{\mathsf{K}}^{\rm resc})_{L^2(S,\infty) \times L^2(S,\infty)} = \pf(J-\overline{\mathcal{A}})_{L^2(S,\infty) \times L^2(S,\infty)}.
\end{equation}
\end{cor}

\begin{proof}
We write the Fredholm expansion as in the proof of Proposition~\ref{prop:cvgOfFredPf}. Then, taking $\kappa> |\delta|$, the bounds of Lemma~\ref{lem:AsymptoticsKernels} allow us to exchange the $N\to\infty$ limit with the sums/integrals by dominated convergence. The result follows.
\end{proof}

\begin{cor}\label{cor:AsymptoticsHs}
For any given $L>0$, the following limits hold uniformly for $Y\in [-L,L]$:
\begin{equation}
\lim_{N\to\infty} h_1^{\rm resc}(Y) = \mathpzc{h}^{\delta,u}_1(Y),\quad
\lim_{N\to\infty} h_2^{\rm resc}(Y) =  \mathpzc{h}^{\delta, u}_2(Y).
\end{equation}
Furthermore, for any $Y\geq -L$ and $\kappa > 0$, we have the following bounds which hold uniformly in $N$:
\begin{equation}
|h_1^{\rm resc}(Y)|\leq C |Y| e^{|\delta Y|},\quad |h_2^{\rm resc}(Y)|\leq C e^{-\kappa Y}
\end{equation}
for some constant $C$.
\end{cor}

\begin{proof}
The bounds of Lemmas~\ref{lem:AsymptoticsFunctions},~\ref{lem:AsymptoticsKernels}, and~\ref{lem:AsymptoticsKernelsTilde} imply that, taking $\kappa> |\delta|$, we can take $N\to\infty$ inside $\int_S^\infty dV \widetilde K^{\rm resc}_{22}(Y,V) f^{-\delta,\rm resc}_+(V)$ and inside $\int_S^\infty dV {\cal E}_1^{\rm resc}(Y,V) f^{-\delta,\rm resc}_+(V)$. Together with the bounds on the remaining terms, we get the stated result.
\end{proof}

With the above results, we are now ready to finish the proof of Theorem~\ref{thm:main_thm_asymp}.

\begin{proof}[Proof of Theorem~\ref{thm:main_thm_asymp}]
The result is now a direct consequence of Corollary~\ref{cor:AsymptoticsFredholmPfaffian}, the bounds and their limits on the rescaled kernel of Lemma~\ref{lem:AsymptoticsKernels}, of Lemma~\ref{lem:AsymptoticsFunctions} for the functions $g_1^{\rm resc}(X)$ and $g_2^{\rm resc}(X)$ entering in left hand side of the scalar product, and of Corollary~\ref{cor:AsymptoticsHs} for the functions $h_1^{\rm resc}(Y)$ and $h_2^{\rm resc}(Y)$ entering in the right hand side of the scalar product. The aforementioned bounds indeed imply that we can take the limit $N\to\infty$ inside the integrals which appear when writing the scalar products explicitly.

The derivatives in $s$ of the kernels and of the functions in the inner products, in fact, produce only polynomial factors, but the bounds are exponential, so by dominated convergence, we have also convergence of these terms to the corresponding derivatives. This leads to the claimed result. We remark that we do not need to worry about the inverse operator, since the derivative acts as
\begin{equation}
\partial_s(\Id-J^{-1}\overline{\mathsf{K}})^{-1}=(\Id-J^{-1}\overline{\mathsf{K}})^{-1} \cdot J^{-1}\partial_s\overline{\mathsf{K}}\cdot (\Id-J^{-1}\overline{\mathsf{K}})^{-1}
\end{equation}
and the resolvent, when multiplied by the Fredholm Pfaffian in front, can be rewritten as a linear combination of two Fredholm Pfaffians, see Remark~\ref{rem:Inverse}.
\end{proof}

\section{Limit to the Baik--Rains distribution: proof of Theorem~\ref{thm:LimitToBR}} \label{sec:limitToBR}

In the $u\to\infty$ limit, we want to take $\delta=-u+\tau$ with $\tau$ fixed. Thus for $u$ large enough we also have $\delta<0$. In this case, i.e.\,for $u>0$ and $\delta<0$, there are some simplifications in the expression of the distribution.

\begin{lem}\label{lem:simplifications}
Consider $u>0$ and $\delta<0$. Then the following equality holds:
\begin{equation}\label{eq:234}
\overline{\mathcal{A}}_{22} (X, Y) = -\int\limits_{-\mu+\I\R} \frac{d \zeta}{2\pi\I} \int\limits_{\mu+\I\R} \frac{d \omega}{2\pi\I} \frac{ e^{\frac{\zeta^3}{3} + \zeta^2 u - \zeta X} }{ e^{\frac{\omega^3}{3} - \omega^2 u - \omega Y} } \frac{1}{\zeta - \omega} \left(\frac{1}{ \zeta + \delta}+\frac{1}{\omega-\delta}\right)
\end{equation}
for any choice of $0<\mu<\min\{-\delta,u\}$ (the contours for $\zeta,\omega$ are oriented with increasing imaginary parts).

Furthermore we have:
\begin{equation}\label{eq:236}
\begin{split}
\widetilde{\mathcal{A}}_{22} (X, Y)+{\cal E}_1(X,Y) =& \overline{\mathcal{A}}_{22}(X,Y)-\mathpzc{g}_2^{\delta, u} (X) e^{-\frac{\delta^3}{3}+\delta^2 u+\delta Y}+ \Id_{X>Y} e^{2\delta^2 u}(e^{\delta(X-Y)}+e^{-\delta(X-Y)}).
\end{split}
\end{equation}
As a consequence of this representation we have:
\begin{equation}\label{eq:237}
\int_S^\infty dV \Id_{Y>V} e^{2\delta^2 u}(e^{\delta(Y-V)}+e^{-\delta(Y-V)}) \mathpzc{f}^{-\delta,u}(V) = \mathpzc{j}^{\delta, u} (S,Y).
\end{equation}

Finally we also have:
\begin{equation}\label{eq:238}
\widetilde{\mathcal{A}}_{12} (X, Y) = \overline{\mathcal{A}}_{12} (X, Y) + \mathpzc{g}_1^{\delta, u} (X) e^{-\frac{\delta^3}{3}+\delta^2 u+\delta Y}.
\end{equation}
\end{lem}
\begin{proof}
First notice that for $\delta<0$, the contours in the double integral of \eqref{eq:227} can be chosen to be the same for the two cases, with $\delta<\Re(\omega)<\Re(\zeta)<-\delta$. Next, notice that we can deform the contours to be vertical provided $\Re(\zeta)>-u$ and $\Re(\omega)<u$. Finally, we exchange the positions of $\zeta$ and $\omega$, so now $\Re(\zeta)<\Re(\omega)$, which is the formula \eqref{eq:234}, minus the pole at $\omega=\zeta$. This pole gives as residue
\begin{equation}
-\int_{\I\R} \frac{d\zeta}{2\pi\I} e^{2\zeta^2 u-\zeta(X-Y)}\frac{2\zeta}{(\zeta+\delta)(\zeta-\delta)}
\end{equation}
which is equal to $-{\cal E}(X,Y)$. To verify this identity, it is enough by anti-symmetry to consider $X>Y$. Extracting the pole at $\zeta=-\delta$ leads to $-{\cal E}_0(X,Y)$, while the remaining integral is $-{\cal E}_1(X,Y)$. Finally, \eqref{eq:237} is an elementary computation and \eqref{eq:238} follows by taking the residue at $\omega=\delta$.
\end{proof}

With the above decomposition we can prove Lemma~\ref{lem:SimplificationBR}.

\begin{proof}[Proof of Lemma~\ref{lem:SimplificationBR}]
Using the representations \eqref{eq:236}---\eqref{eq:238}, our claim holds if we can show that
\begin{equation}
\braket{-\mathpzc{g}_1^{\delta, u} \quad \mathpzc{g}_2^{\delta, u} } { (\Id-J^{-1}\overline{\mathcal{A}})^{-1} \begin{pmatrix} \mathpzc{g}_2^{\delta, u}  \braket{\mathpzc{f}^{-\delta,-u}}{\mathpzc{f}^{-\delta,u}} \\ \mathpzc{g}_1^{\delta, u} \braket{\mathpzc{f}^{-\delta,-u}}{\mathpzc{f}^{-\delta,u}} \end{pmatrix} }=0.
\end{equation}
The proof of this is the same as proving that $\eqref{eq:diff_inner}=0$ in Lemma~\ref{lem:GbarToGhat}. Notice that for $\delta<0$ (but not for $\delta\geq 0$) the scalar product $\braket{\mathpzc{f}^{-\delta,-u}}{\mathpzc{f}^{-\delta,u}}$ is well-defined.
\end{proof}

In order to analyze the $u\to\infty$ limit with $u+\delta=\tau$ constant, we need to consider a conjugation in the kernel entries, but also to shift the positions by $\delta(2u+\delta)$ as discussed above in Remark~\ref{rem:simple_scaling}. Finally, to clearly see the limit $u\to\infty$, we shift the $\zeta,\omega$ integration variables to remove the $\zeta^2,\omega^2$ terms in the exponential.

\begin{lem} \label{lem:A_conj}
Let us consider $u>0$, $\delta<0$ and $u+\delta=\tau$. Shifting the positions as $X=x+\delta(2u+\delta)$ and $Y=y+\delta(2u+\delta)$, we have:
\begin{align} \label{eq:A_conj}
 \frac{e^{\frac23 u^3+u X}}{e^{-\frac23 u^3-u Y}} \overline{\mathcal{A}}_{11} (X, Y) &= - \int\limits_{{ }_{-u}\zcd} \frac{d z}{2\pi\I} \int\limits_{\wcu\, {}_{u,z+2u}} \frac{d w}{2\pi\I}\frac{ e^{\frac{z^3}{3} - z(x+\tau^2)} }{ e^{\frac{w^3}{3} - w(y+\tau^2)} } \frac{(w+z) (w+\tau -2u) (z-\tau +2 u)}{4 (w-u) (z+u) (z-w+2 u)}, \nonumber \\
\frac{e^{\frac23 u^3+u X}}{e^{\frac23 u^3+u Y}} \overline{\mathcal{A}}_{12} (X, Y) &= -\int\limits_{{}_{-u}{\zcd}} \frac{d z}{2\pi\I} \int\limits_{{}_{\tau-2u}\wcu\, {}_z} \frac{d w}{2\pi\I} \frac{ e^{\frac{z^3}{3}- z (x+\tau^2)} }{ e^{\frac{w^3}{3} - w (y+\tau^2)} } \frac{(w+z+2 u) (z-\tau +2 u)}{2 (z+u) (z-w) (w-\tau +2 u)}, \\
\frac{e^{-\frac23 u^3-u X}}{e^{\frac23 u^3+u Y}} \overline{\mathcal{A}}_{22} (X, Y) &= -\int\limits_{r+\I\R} \frac{d z}{2\pi\I} \int\limits_{r+\I\R} \frac{d w}{2\pi\I} \frac{ e^{\frac{z^3}{3}- z(x+\tau^2)} }{ e^{\frac{w^3}{3}  - w(y+\tau^2)} } \frac{z+w}{(w-\tau +2 u) (z-w-2 u) (z+\tau -2 u)} \nonumber
\end{align}
where for $\overline{\mathcal{A}}_{22}$ the integration contours for $z,w$ are oriented with increasing imaginary part and $0<r<\min\{u,2u-\tau\}$.
\end{lem}

\begin{proof}
The first equality is obtained by the change of variables $\zeta=z+u$ and $\omega=w-u$; the second is obtained by the change of variables $\zeta=z+u$ and $\omega=w+u$; finally, the third comes from substituting $\zeta=z-u$ and $\omega=w+u$ in the representation \eqref{eq:234}.
\end{proof}

We have a similar result for $\mathpzc{e}_1^{\delta, u}$, $\mathpzc{g}_1^{\delta, u}, \mathpzc{g}_2^{\delta, u}, \tilde{\mathpzc{h}}_1^{\delta, u}$, and $\tilde{\mathpzc{h}}_2^{\delta, u}$.

\begin{lem} \label{lem:gh_conj}
Let us consider $u>0$, $\delta<0$ and $u+\delta=\tau$. Shifting the positions as $X=x+\delta(2u+\delta)$, $Y=y+\delta(2u+\delta)$, and $S=s+\delta(2u+\delta)$, we have:
\begin{equation}
\begin{split}
\mathpzc{e}^{\delta, u} (S) =& \,\,{\cal R}_{-\tau}(s),  \\
 e^{\frac23 u^3+u X} \mathpzc{g}_1^{\delta, u} (X) =& \int\limits_{{}_{-u}\zcd} \frac{dz}{2 \pi \I} e^{\frac{z^3}{3} - z (x+\tau^2)} \frac{z+\tau}{2 (z+u)}, \\
 e^{-\frac23 u^3 - u X} \mathpzc{g}_2^{\delta, u} (X) =& \int\limits_{\zcd\, {}_\tau} \frac{dz}{2 \pi \I} e^{\frac{z^3}{3} - z (x+\tau^2)} \frac{1}{z-\tau}
\end{split}
\end{equation}
as well as
\begin{equation}
\begin{split}
 e^{-\frac23 u^3 - u Y} \tilde{\mathpzc{h}}_1^{\delta, u} (Y) =& \ e^{\frac23 \tau^3 + s \tau} \!\!\! \int\limits_{r + \I \R} \!\!\! \frac{d z}{2\pi\I} \!\!\!\! \int\limits_{ -r + \I \R} \!\!\!\! \frac{d w}{2\pi\I} \frac{ e^{\frac{z^3}{3} - z (y+\tau^2)} }{ e^{\frac{w^3}{3} - w (s+\tau^2)} } \frac{(z \! + \! w)} {(w \! - \! \tau \! + \! 2 u) (z \! - \! w \! - \! 2 u) (z \! + \! \tau \! - \! 2 u) (w \! + \! \tau)} \\
 & - \!\!\! \int\limits_{\zcd\, {}_{\tau, 2 u - \tau}} \!\!\! \frac{dz}{2 \pi \I} e^{\frac{z^3}{3} - z (y+\tau^2)} \frac{2 (z-u)} {(z-\tau) (z-2 u + \tau)}, \\
 e^{\frac23 u^3 + u Y} \tilde{\mathpzc{h}}_2^{\delta, u} (Y) =& \ e^{\frac23 \tau^3 + s \tau} \int\limits_{{}_{-u}\zcd } \frac{d z}{2\pi\I} \!\!\! \int\limits_{{}_{\tau-2u}\wcu\, {}_{z,-\tau}} \!\!\! \frac{d w}{2\pi\I} \frac{ e^{\frac{z^3}{3} - z (y+\tau^2)} }{ e^{\frac{w^3}{3} - w (s+\tau^2)} } \frac{(w+z+2 u) (z-\tau +2 u)}{2 (z+u) (z-w) (w-\tau +2 u) (w + \tau)} \\
 & + \int\limits_{{}_{-\tau}\zcd } \frac{dz}{2 \pi \I} e^{\frac{z^3}{3} - z (y+\tau^2)} \frac{1}{z+\tau}
\end{split}
\end{equation}
where $0 < r < \min \{u, 2u-\tau\}$.
\end{lem}

\begin{proof}
The proofs for $\mathpzc{e}^{\delta, u}$, $\mathpzc{g}_1^{\delta, u}$ and $\mathpzc{g}_2^{\delta, u}$ consist, as above, respectively in the following changes of variables: $\zeta = z-u$, $\zeta = z+u$ and $\zeta = z-u$. The last two also work for $\mathpzc{g}_3^{\delta, u}$ and $\mathpzc{g}_4^{\delta, u}$ respectively, as summands for $\tilde{\mathpzc{h}}_2^{\delta, u}$ and $\tilde{\mathpzc{h}}_1^{\delta, u}$:
\begin{equation} \label{eq:g34_conj}
\begin{split}
 e^{\frac23 u^3+u Y} \mathpzc{g}_3^{\delta, u} (Y) &= \int\limits_{{}_{-\tau}\zcd} \frac{dz}{2 \pi \I} e^{\frac{z^3}{3} - z (y+\tau^2)} \frac{1}{z+\tau}, \\
 e^{-\frac23 u^3 - u Y} \mathpzc{g}_4^{\delta, u} (Y) &= \!\!\! \int\limits_{\zcd\, {}_{\tau, 2 u - \tau}} \!\!\! \frac{dz}{2 \pi \I} e^{\frac{z^3}{3} - z (y+\tau^2)} \frac{2 (z-u)} {(z-\tau) (z-2 u + \tau)}.
 \end{split}
 \end{equation}
For the term $\overline{\mathcal{A}}_{12} \mathpzc{f}^{-\delta, u}$ we have:
\begin{equation}
\begin{split}
 & e^{\frac23 u^3+u Y} \left( \overline{\mathcal{A}}_{12} \mathpzc{f}^{-\delta, u} \right) (Y) = \\
 & \qquad \qquad \qquad = \ e^{\frac23 \tau^3 + s \tau} \int\limits_{{}_{-u}\zcd } \frac{d z}{2\pi\I} \!\!\! \int\limits_{{}_{\tau-2u}\wcu\, {}_{z,-\tau}} \!\!\! \frac{d w}{2\pi\I} \frac{ e^{\frac{z^3}{3} - z (y+\tau^2)} }{ e^{\frac{w^3}{3} - w (s+\tau^2)} } \frac{(w+z+2 u) (z-\tau +2 u)}{2 (z+u) (z-w) (w-\tau +2 u) (w + \tau)}
 \end{split}
\end{equation}
which can be obtained by explicitly performing the integration $\int_S^\infty dV e^{(\delta+\omega) V}$ in the product and then changing variables $(\zeta,\omega) = (z+u, w+u)$. The computation for $e^{-\frac23 u^3-uY}\overline{\mathcal{A}}_{22} \mathpzc{f}^{-\delta, u}$ is similar. Combining them with the $\mathpzc{g}_3^{\delta, u}$ term, respectively the $\mathpzc{g}_4^{\delta, u}$ term, we obtain the result for $\tilde{\mathpzc{h}}_2^{\delta, u}$ and respectively $\tilde{\mathpzc{h}}_1^{\delta, u}$.
\end{proof}

\begin{proof}[Proof of Theorem~\ref{thm:LimitToBR}]
The finiteness of the Fredholm Pfaffian and of the scalar products depends on the behavior in $x,y$ in the above expressions. The $u$ dependence is only marginal and, with the chosen conjugation, all the terms remain bounded as $u \to \infty$. By dominated convergence we can take the $u \to \infty$ limit inside both the Fredholm Pfaffians and the scalar product. We have the following limits:
\begin{equation}
 \lim_{u \to \infty} \frac{e^{\frac23 u^3+u X}}{e^{-\frac23 u^3-u Y}} \overline{\mathcal{A}}_{11} (X, Y) = \lim_{u \to \infty} \frac{e^{-\frac23 u^3-u X}}{e^{\frac23 u^3+u Y}} \overline{\mathcal{A}}_{22} (X, Y) = 0
\end{equation}
and
\begin{equation}
 \lim_{u \to \infty} \frac{e^{\frac23 u^3 + u X}}{e^{\frac23 u^3 + u Y}} \overline{\mathcal{A}}_{12} (X, Y) = \mathpzc{K}_{\ \rm Ai,\tau} (x,y).
\end{equation}
Dominated convergence then implies that
\begin{equation}
 \lim_{u \to \infty} \pf(J - \overline{\mathcal{A}})_{L^2(S, \infty) \times L^2(S,\infty)} = \det(\Id - \mathpzc{K}_{\ \rm Ai,\tau})_{L^2 (s, \infty)} = F_{\rm GUE}(s+\tau^2)
\end{equation}
and that
\begin{equation}
 \lim_{u \to \infty} J^{-1} \overline{\mathcal{A}} (X, Y) = \begin{pmatrix} \mathpzc{K}_{\ \rm Ai,\tau}(x, y) & 0 \\ 0 & \mathpzc{K}_{\ \rm Ai,\tau}(x,y) \end{pmatrix}.
\end{equation}
The latter limit extends to resolvents as well.

The $\mathpzc{g}$ and  $\tilde{\mathpzc{h}}$ functions have the following limits:
\begin{equation}
 \lim_{u \to \infty} e^{\frac23 u^3+u X} \mathpzc{g}_1^{\delta, u} (X) = 0, \quad \lim_{u \to \infty} e^{-\frac23 u^3 - u X} \mathpzc{g}_2^{\delta, u} (X) =  \Psi_{-\tau}(x)
\end{equation}
and
\begin{equation}
  \begin{aligned}
\lim_{u \to \infty} e^{-\frac23 u^3 - u Y} \tilde{\mathpzc{h}}_1^{\delta, u} (Y) &= -\Psi_{-\tau}(y), \\
\lim_{u \to \infty} e^{\frac23 u^3 + u Y} \tilde{\mathpzc{h}}_2^{\delta, u} (Y) &= e^{\frac23 \tau^3 + s \tau} \int\limits_{\zcd} \frac{d z}{2\pi\I} \int\limits_{\wcu\, {}_{z,-\tau}} \frac{d w}{2\pi\I} \frac{ e^{\frac{z^3}{3} - z (y+\tau^2)} }{ e^{\frac{w^3}{3} - w (s+\tau^2)} } \frac{1}{(z-w) (w + \tau)} \\
 & \quad + \int\limits_{{}_{-\tau}\zcd} \frac{dz}{2 \pi \I} e^{\frac{z^3}{3} - z (y+\tau^2)} \frac{1}{z+\tau} =  \Phi_{-\tau}(y)
 \end{aligned}
\end{equation}
where the last equality is obtained by computing the residue at $w=-\tau$.

To conclude, the inner product on the right of~\eqref{eq:final_dist_asymp} collapses, in the limit $u \to \infty$, to $\braket{\Psi_{-\tau}}{ (\Id - \mathpzc{K}_{\ \rm Ai,\tau})^{-1} \Phi_{-\tau} }$: as $\mathpzc{g}_{1}^{\delta, u} \to 0$ with $\tilde{\mathpzc{h}}_1^{\delta, u}$ staying finite, the respective inner product is zero in the limit; moreover $\mathpzc{g}_{2}^{\delta, u}, \tilde{\mathpzc{h}}_2^{\delta, u}$ and $J^{-1} \overline{\mathcal{A}}$ converge to the desired quantities. Combining with $\mathpzc{e}^{\delta, u} \to \mathcal{R}_{-\tau}$, we conclude that
\begin{equation}
\lim_{u\to\infty} F^{(\delta, u)}_{0,\,\mathrm{half}} (S) = F_{{\rm BR},-\tau}(s) = F_{{\rm BR},\tau}(s)
\end{equation}
with the last equality coming from the fact that the Baik--Rains distribution is symmetric under $\tau \to -\tau$.
\end{proof}

\appendix

\section{On Pfaffians and point processes} \label{sec:pfaff}

Given an anti-symmetric $2n \times 2n$ matrix $(a_{i,j})$, its Pfaffian is defined as:
\begin{equation}
 \pf [a_{i,j}]_{1\leq i < j\leq 2n} = \frac{1}{2^{n} n!}\sum_{\sigma \in S_{2n}}\sgn(\sigma)a_{\sigma(1),\sigma(2)} a_{\sigma(3),\sigma(4)} \cdots a_{\sigma(2n-1),\sigma(2n)}
\end{equation}
where $S_{2n}$ is the permutation group on $2n$ letters. Remark that the Pfaffian is determined entirely by the upper triangular part. One can show that \begin{equation}\label{eq:pf_det}
\left(\pf[a_{i,j}]_{1\leq i < j\leq 2n}\right)^2=\det [a_{i,j}]_{1\leq i,j\leq 2n}.
\end{equation}

Suppose that one has a $2 \times 2$ anti-symmetric matrix kernel $K(x, y)$, i.e.\,$K$ is a $2 \times 2$ matrix function of $(x, y)$ satisfying $K(x, y) = -K^t(y, x)$ with $t$ denoting transposition. Notice the interchange of $x$ and $y$. Given such a kernel and points (variables or numbers) $x_1, \dots, x_n$, one can define a $2n \times 2n$ anti-symmetric matrix $K^{(n)}$ block-wise as follows: its $2 \times 2$ block $(i,j)$ for $1 \leq i, j \leq n$ is given by the matrix $K(x_i, x_j)$. Because $K(x, y) = -K^t(y, x)$, $K^{(n)}$ thus defined is even-dimensional anti-symmetric and its Pfaffian well-defined.

Throughout we are interested in probability distributions on configuration spaces $\Lambda$ equipped with a measure $d \mu$ as follows: either $\Lambda = \Z$ (or a semi-infinite interval in $\Z$) is discrete and $d \mu$ is counting measure, or $\Lambda = \R$ (or a semi-infinite interval in $\R$) is continuous and $d \mu = d x$ is Lebesgue measure. A point process (measure)\footnote{For more on point processes in general and determinantal ones in particular, see e.g.\,\cite{Jo05,BOO00,Bor10}} on such a configuration space $\Lambda$ is called \emph{Pfaffian with $2 \times 2$ matrix correlation kernel $K$} if there exists a $2 \times 2$ matrix kernel $K$ satisfying $K(x, y) = -K^t(y, x)$ such that
the $n$-point correlation functions $\rho_n(x_1, x_2, \dots, x_n)$ of the process, for all $n \geq 1$, are Pfaffians of the associated $2n \times 2n$ matrix $K^{(n)}$:
\begin{equation}
 \rho_n(x_1, x_2, \dots, x_n) = \pf[K^{(n)} (x_i, x_j)]_{1\leq i<j\leq n}.
\end{equation}
We recall that if $\Lambda$ is discrete, $\rho_n(x_1, x_2, \dots, x_n) = \Pb(S : x_i \in S,\ \forall\,1 \leq i \leq n)$ is the probability of $S$ containing all the $x$'s; if $\Lambda$ is continuous, $\rho_n(x_1, x_2, \dots, x_n) d x_1 \dots d x_n$ is the probability of sets $S$ having non-empty intersection with all $[x_i, x_i + d x_i]$. A simple observation states that the one-point function is the $K_{12}$ entry: $\Pb(S:x \in S) = \rho_1(x) = K_{12}(x, x)$ (in the discrete setting; the obvious modification is needed in the continuous setting).

Given a $2 \times 2$  anti-symmetric matrix kernel $K$ defined on a configuration space $\Lambda$ equipped with a measure $d \mu$ (either counting measure or Lebesgue measure depending on the underlying space), the \emph{Fredholm Pfaffian} of $K$ restricted to the subspace $U \subset \Lambda$ is defined as
 \begin{equation}\label{eq:fredholm_pf_def}
\pf(J + \lambda K)_{L^2(U) \times L^2(U)} = \sum_{n=0}^{\infty}\frac{\lambda^n}{n!} \int_{U^n} \pf[ K^{(n)} (x_i, x_j)]_{1\leq i<j\leq n} \prod_{i=1}^n d \mu(x_i)
\end{equation}
assuming the sum is finite. Here $J$ is the anti-symmetric matrix kernel $J(x,y)=\delta_{x,y} \left( \begin{smallmatrix} 0 & 1 \\ -1 & 0 \end{smallmatrix} \right)$, but as is oftentimes the case in the literature, this technicality is overlooked and we think of $J$ just as the corresponding $2 \times 2$ matrix. We note that oftentimes a sufficient condition for the Fredholm Pfaffian $\pf(J+\lambda K)_{L^2(U) \times L^2(U)}$ to be finite finite is that $K$ is a trace class operator on $L^2(U) \times L^2(U)$ (or rather it has trace class entries). Moreover, Fredholm Pfaffians are defined up to conjugation, in the following sense. Suppose $\tilde{K}$ is the anti-symmetric matrix kernel
\begin{equation}
\tilde{K}(x,y)=
\left(\begin{smallmatrix} e^{f(x)} & 0 \\ 0 & e^{-f(x)} \end{smallmatrix}\right)
K(x,y)
\left(\begin{smallmatrix} e^{f(y)} & 0 \\ 0 & e^{-f(y)} \end{smallmatrix}\right)
\end{equation}
for a $d\mu$-measurable function $f$. Then it is not hard to check that $\pf [K^{(n)} (x_i, x_j)]_{1\leq i<j\leq 2n} = \pf [\tilde{K}^{(n)} (x_i, x_j)]_{1\leq i<j\leq 2n}$ and so $\pf(J+\lambda K) = \pf(J+\lambda \tilde{K})$ whenever both are defined. Importantly, we can use this to define $\pf(J+\lambda K)$ even if $K$ is not a trace class operator provided we find an appropriate $f$ which makes $\tilde{K}$ trace class.

Owing to identity~\eqref{eq:pf_det}, we have the following relation between Fredholm Pfaffians with $2 \times 2$ matrix kernels $K$ and block Fredholm determinants with related kernel $J^{-1}K$:
\begin{equation}\label{eq:fred_pf_det}
  \pf(J + \lambda K)^2_{L^2(U) \times L^2(U)} = \det(\Id + \lambda J^{-1} K)_{L^2(U) \times L^2(U)}
\end{equation}
where we remark the Fredholm determinant on the right hand side is defined as in~\eqref{eq:fredholm_pf_def} with pf replaced by det and $K^{(n)}$ by $(J^{-1}K)^{(n)}$.

Finally, for any two bounded $2 \times 2$ matrix operators $A:L^2(U) \times L^2(U) \to L^2(V) \times L^2(V)$, $B:L^2(V) \times L^2(V) \to L^2(U) \times L^2(U)$ for which both sides below are defined, we have
  \begin{equation} \label{eq:pfaff_id}
   \pf(J+\lambda JAB)_{L^2(V) \times L^2(V)}=\pf(J+\lambda JBA)_{L^2(U) \times L^2(U)}.
  \end{equation}
For example, this could happen if the entries of either $A$ or $B$ are trace class (as such operators form an ideal inside the bounded ones), or if the entries of both $A$ and $B$ are Hilbert--Schmidt (as products of two Hilbert--Schmidt operators are trace class), but it could happen in more generality too. We note that oftentimes $AB$ is infinite dimensional while $BA$ is finite dimensional. This fact is immediately implied by the corresponding Fredholm determinant identity $\det(\Id + \lambda AB) =\det(\Id + \lambda BA)$ and~\eqref{eq:fred_pf_det}.

For more on Pfaffians, the reader is referred to the the Appendix of~\cite{OQR16} for the analytic side and to~\cite{Ste90} for the algebraic and combinatorial ones.

\section{Some determinantal computations} \label{sec:det_comp}

In this section we prove~\eqref{eq:F_det}. For this we need to show that
\begin{equation}
\braket{Y_1}{Z_2} = \braket{Y_2}{Z_1}=0 \text{ and } \braket{Y_1}{Z_1}= \braket{Y_2}{Z_2}
\end{equation}
where $Z_i, Y_i, i=1,2$ are defined in~\eqref{eq:S_fact}. For brevity and consistency, let us rename
\begin{equation}
  A=(\Id-\overline G)^{-1}, \quad h_1= g_1, \quad h_2=g_2, \quad f_1 = f_{+}^\beta.
\end{equation}

We have
\begin{equation}\label{eq:inner1}
 \braket{Y_1}{Z_2} = \braket{f_1}{A_{12} f_1}
\end{equation}
and
\begin{equation}\label{eq:inner2}
 \braket{Y_2}{Z_1} = \braket{-h_1 \quad h_2} {\begin{array}{c} A_{11}h_2+A_{12}h_1 \\ A_{21}h_2+A_{22} h_1 \end{array}}.
\end{equation}
To show that both are zero, we need the following result.
\begin{prop}\label{prop:kernel_symm}
We have:
\begin{equation}
 \begin{aligned}
  A_{12}(x,y)=-A_{12}(y,x),\quad A_{21}(x,y)=-A_{21}(y,x),\quad A_{11}(x,y)=A_{22}(y,x).
 \end{aligned}
\end{equation}
\end{prop}
\begin{proof}
 Assume for now that we know the norm of $\overline{G}$ is less than 1. Then we can expand $A$ as a Neumann series $(\Id-\overline G)^{-1}=\Id+\overline G+\overline G^2+\overline G^3+\dots$. We will prove by induction that, for any $n\geq1$
 \begin{equation}\label{eq:induction}
  \begin{aligned}
  \overline G^n_{12}(x,y)=-\overline G^n_{12}(y,x),\quad \overline G^n_{21}(x,y)=-\overline G^n_{21}(y,x),\quad \overline G^n_{11}(x,y)=\overline G^n_{22}(y, x).
  \end{aligned}
 \end{equation}
The base case is true by the definition of $\overline G=J^{-1}\overline K$ and by the fact that
 \begin{equation}
  \begin{aligned}
  \overline K_{12}(x,y)=-\overline K_{21}(y,x),\quad \overline K_{11}(x,y)=-\overline K_{11}(y,x),\quad \overline K_{22}(x,y)=-\overline K_{22}(y,x).
  \end{aligned}
 \end{equation}
Now we proceed with the inductive step: assuming that~\eqref{eq:induction} holds for $B=\overline G^n$, for $C=\overline G^{n+1}$ we have
\begin{equation}
 \begin{aligned}
  C_{11}(x,y)&=\left(B_{11}\overline G_{11}+B_{12}\overline G_{21}\right)(x,y)\\
  &=\int_s^{\infty} dz B_{11}(x,z)\overline G_{11}(z,y)+\int_s^{\infty} dz B_{12}(x,z)\overline G_{21}(z,y)\\
  &=\int_s^{\infty} dz B_{22}(z,x)\overline G_{22}(y,z)+\int_s^{\infty} dz B_{12}(z,x)\overline G_{21}(y,z)\\
  &=C_{22}(y,x).
 \end{aligned}
\end{equation}
Moreover,
\begin{equation}
 \begin{aligned}
  C_{12}(x,y)&=\left(B_{11}\overline G_{12}+B_{12}\overline G_{22}\right)(x,y)\\
  &=\int_s^{\infty} dz B_{11}(x,z)\overline G_{12}(z,y)+\int_s^{\infty} dz B_{12}(x,z)\overline G_{22}(z,y)\\
  &=-\int_s^{\infty} dz B_{22}(z,x)\overline G_{12}(y,z)-\int_s^{\infty} dz B_{12}(z,x)\overline G_{11}(y,z)\\
  &=-C_{12}(y,x)
 \end{aligned}
\end{equation}
where we have used that $C = B \overline{G} = \overline{G} B$. An analogous computation holds for $C_{21}$.

Now it may happen that the norm of $\overline{G}$, of which we remain ignorant, is $\geq 1$. We then replace $\overline{G}$ by $\omega \overline{G}$ for $\omega>0$ small enough to make $\omega \overline{G}$ of subunit norm. The argument above applies to the corresponding $A(\omega) = (\Id-\omega \overline{G})^{-1}$. We then analytically continue in $\omega$, as the spectrum of any trace class operator (and in particular of $\overline{G}$) is discrete without non-zero accumulation points.
\end{proof}
Using Proposition~\ref{prop:kernel_symm} in equation~\eqref{eq:inner1}, we have
\begin{equation}
 \begin{aligned}
  \braket{Y_1}{Z_2} &=\int_s^{\infty} \int_s^{\infty} dx dy f_1(x)A_{12}(x,y)f_1(y)\\
  &=-\int_s^{\infty} \int_s^{\infty} dx dy f_1(x)A_{12}(y,x)f_1(y) \\
  &=0.
 \end{aligned}
\end{equation}
For the same reason,~\eqref{eq:inner2} becomes:
\begin{equation}
  \braket{Y_1}{Z_2} = -\braket{h_1}{A_{11}h_2} - \braket{h_1}{A_{12}h_1} + \braket{h_2}{A_{22}h_1} + \braket{h_2}{A_{21}h_2} = 0
\end{equation}
since the last two terms are zero, and the first two are equal by Proposition~\ref{prop:kernel_symm}:
\begin{equation}
 \begin{aligned}
  \braket{h_2}{A_{22}h_1} &= \int_s^{\infty} \int_s^{\infty} dx dy h_2(x)A_{22}(x,y)h_1(y)\\
  &=\int_s^{\infty} \int_s^{\infty} dx dy h_2(x)A_{11}(y,x)h_1(y)\\
  &=\braket{h_1}{A_{11}h_2}.
 \end{aligned}
\end{equation}
Now, we show that $\braket{Y_1}{Z_1} = \braket{Y_2}{Z_2}$. The two are explicitly given by
\begin{equation}
 \braket{Y_1}{Z_1} = \braket{f_1}{A_{11} h_2} + \braket{f_1}{A_{12}h_1}, \quad \braket{Y_2}{Z_2} = -\braket{h_1}{A_{12} f_1} + \braket{h_2}{A_{22}f_1}.
\end{equation}
By Proposition~\ref{prop:kernel_symm}, as $A_{11}(x, y) = A_{22}(y, x)$ and $A_{12}(x, y) = - A_{12}(y, x)$, it follows that $\braket{f_1}{A_{11} h_2} = \braket{h_2}{A_{22} f_1}$ and that $\braket{f_1}{A_{12} h_1} = -\braket{h_1}{A_{12} f_1}$ proving the desired equality.

\section{Correlations for geometric weights} \label{sec:geom_wts}

\subsection{Correlation kernel for generic geometric weights}

In this section we briefly review the main result on integrable last passage percolation with \iid geometric random variables. It is by passing to the exponential limit in this result that we obtain Theorem~\ref{thm:exp_corr}.

Consider half-space last passage percolation with \iid\,geometric random variables $(W_{i,j})_{1 \leq j \leq i \leq N}$ distributed as follows:
\begin{equation}
 W_{i,j} = \begin{cases}
            \mathrm{Geom}(a x_i), & i = j, \\
            \mathrm{Geom}(x_i x_j), & i > j
           \end{cases}
\end{equation}
where $a, x_1, \dots, x_N$ are real parameters satisfying
\begin{equation}
0 \leq a < \min_i \tfrac{1}{x_i}, \quad 0 < x_1, \dots, x_N < 1
\end{equation}
and a random variable $X$ is \textit{geometric} $\mathrm{Geom}(q)$ if $\Pb(X = k) = (1-q)q^k, \forall k \in \N$. We depict this in Figure~\ref{fig:geom_lpp}. It is helpful to visualize the $x$'s as indexing the rows and columns of the half-space.

\begin{figure}[t!]
  \centering
   \includegraphics[height=5.5cm]{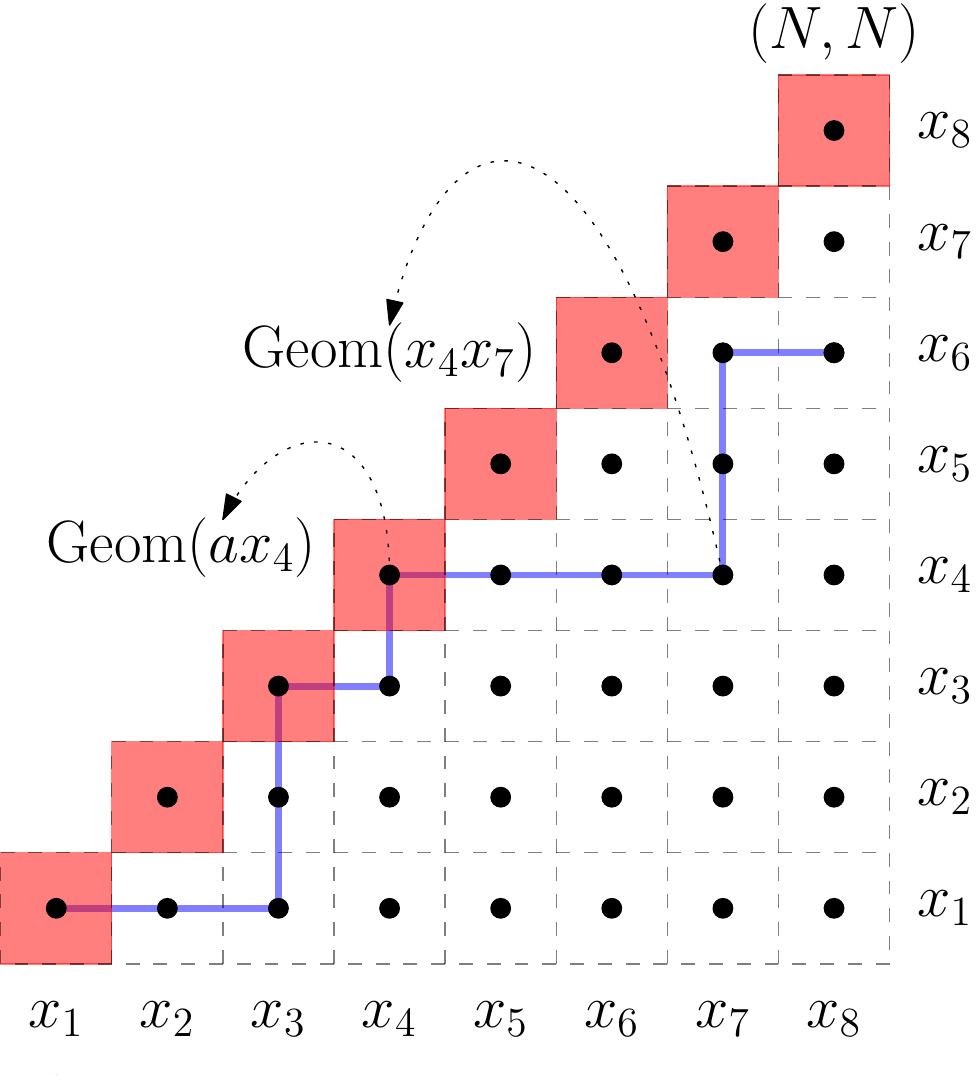}
   \caption{A possible LPP path (polymer) from $(1, 1)$ to $(N, N-n)$ for $(N, n) = (8, 2)$. The non-diagonal dots are independent geometric random variables with parameters assigned by multiplying the row and column $x$ parameters; the diagonal has an extra parameter $a$.}
   \label{fig:geom_lpp}
\end{figure}

We have the following result for the last passage percolation time $L_{N,N-n}$ from $(1, 1)$ to \mbox{$(N,N-n)$} (for fixed $0 \leq n \leq N-1$). The precise statement as stated below is first rigorously proven and explicitly stated in~\cite{BBCS17}. Previously the kernel was derived in~\cite{SI03}, but not in a completely rigorous way, specifically for the $n=0$ case. The Pfaffian structure for $n=0$ was first derived by Rains~\cite{Ra00}, and subsequently extended for generic $n$ and/or as multi-point joint distribution in~\cite{SI03, RB04, BBCS17, Gho17, BBNV18}. See also the pioneering algebraic work of Baik--Rains~\cite{BR99, BR01b} for an alternative but equivalent approach via Toeplitz+Hankel determinants and matrix integrals, as well as the more combinatorial approach of Forrester--Rains~\cite{FR07}. We follow~\cite{BBNV18} for the exposition hereinafter.

\begin{thm} \label{thm:geom_corr}
 The distribution of $L_{N,N-n}$ is a Fredholm Pfaffian
 \begin{equation}\label{eqC3}
  \Pb(L_{N,N-n} \leq s) = \pf (J-K^{\rm geo})_{\ell^2 (\{s+1, s+2,\dots\})}
 \end{equation}
 with $2 \times 2$ matrix correlation kernel $K^{\rm geo}: \Z^2 \to \mathrm{Mat}_2(\R)$ given by
 \begin{equation}
  \begin{split}
   K^{\rm geo}_{11}(k, \ell) &= \ics \int \frac{dz}{z^{k}} \int \frac{dw}{w^{\ell}} F(z) F(w) \frac{(z-w) (z-a) (w-a)} {(z^2-1) (w^2-1) (zw-1)}, \\
   K^{\rm geo}_{12}(k, \ell) = -K^{\rm geo}_{21}(\ell, k) &= \ics \int \frac{dz}{z^{k}} \int \frac{dw}{w^{-\ell+1}} \frac{F(z)} {F(w)} \frac{(zw-1) (z-a)}{(z-w) (z^2-1) (w-a)}, \\
   K^{\rm geo}_{22}(k, \ell) &= \ics \int \frac{dz}{z^{-k+1}} \int \frac{dw}{w^{-\ell+1}} \frac{1}{F(z) F(w)} \frac{(z-w)}{ (zw-1) (z-a) (w-a)}
  \end{split}
 \end{equation}
 where
 \begin{equation}
  F(z)= \frac{\prod_{k=1}^N (1-x_k/z)}{\prod_{k=1}^{N-n} (1-x_k z)}
 \end{equation}
and where the contours are positively oriented circles centered around the origin satisfying the following conditions:
  \begin{itemize}
   \item for $K^{\rm geo}_{11}$, $1 < |z|, |w| < \min_i \frac{1}{x_i}$;
   \item for $K^{\rm geo}_{12}$, $\max \{ \max_i x_i, a\} < |w| < |z|$ and $1 < |z| < \min_i \frac{1}{x_i}$;
   \item for $K^{\rm geo}_{22}$, $\max \{ \max_i x_i, a\} < |w|, |z|$ and $1 < |zw|$.
  \end{itemize}
\end{thm}

\subsection{Reformulation of the correlation kernel for our model}

Now consider the case $x_1=b\in (0,1)$, $x_2=x_3=\ldots=x_N=\sqrt{q}$ for some $q\in (0,1)$ and $a>0$ such that $a\sqrt{q}<1$ and $ab<1$. We will manipulate the kernel so that the limit to the exponential case becomes straightforward. In particular, it is not necessary to do any steepest descent analysis (Laplace method) as was done in~\cite{BBCS17}.

Let us define the following functions:
\begin{equation}
H(z)=\frac{(1-\sqrt{q}/z)^{N-1}}{(1-\sqrt{q}z)^{N-1}},\quad B(z)= \frac{1-b/z}{1-bz}
\end{equation}
which when combined yield
\begin{equation}
F(z)=H(z) B(z) (1-\sqrt{q} z)^n.
\end{equation}
Notice that
\begin{equation}
H(z^{-1}) = H(z)^{-1},\quad B(z^{-1})=B(z)^{-1},\quad F(z^{-1})=H(z)^{-1} B(z)^{-1} (1-\sqrt{q}/z)^n.
\end{equation}
Since in the Fredholm Pfaffian~\eqref{eqC3} the smallest value to be considered is $s=0$, the $k,\ell$ entries of the kernel will be in $\{1,2,3,\ldots\}$ only. Below we only consider the case $b\neq \sqrt{q}$, since the $b=\sqrt{q}$ case is simpler: just set $B(z)=0$ and replace $N-1$ with $N$ in $H(z)$. To see this, take for instance the $b\to\sqrt{q}$ limit in~\eqref{eqC13}.

\begin{lem}
  Let $b\neq \sqrt{q}$ and $k,\ell\geq 1$. The kernel entry $K^{\rm geo}_{11}$ can be expressed as
\begin{equation}\label{eqC9}
\begin{aligned}
K^{\rm geo}_{11}(k,\ell) = &\frac{-1}{(2\pi\I)^2}\oint\limits_{\Gamma_{\sqrt{q}}} dw\!\!\! \oint\limits_{\Gamma_{1/\sqrt{q}}}\!\!\! dz \frac{w^{\ell-1}}{z^{k}}\frac{h^{\rm geo}_{11}(z,w)}{z-w} +
\frac{-1}{(2\pi\I)^2}\oint\limits_{\Gamma_{b}} dw\!\!\! \oint\limits_{\Gamma_{1/\sqrt{q}}}\!\!\! dz \frac{w^{\ell-1}}{z^{k}}\frac{h^{\rm geo}_{11}(z,w)}{z-w}\\
&+\frac{-1}{(2\pi\I)^2}\oint\limits_{\Gamma_{\sqrt{q}}} dw \oint\limits_{\Gamma_{1/b}} dz \frac{w^{\ell-1}}{z^{k}}\frac{h^{\rm geo}_{11}(z,w)}{z-w}
\end{aligned}
\end{equation}
with
\begin{equation}
h^{\rm geo}_{11}(z,w)=\frac{H(z)B(z)}{H(w)B(w)}(1-\sqrt{q} z)^n(1-\sqrt{q}/w)^n \frac{(zw-1)(z-a)(1-w a)}{(z^2-1)(1-w^2)}.
\end{equation}
\end{lem}

\begin{proof}
Let us start with
\begin{equation}
K^{\rm geo}_{11}(k,\ell) = \oint \frac{dw}{2\pi\I} \oint \frac{dz}{2\pi\I} \frac{H(z)(1-\sqrt{q} z)^nB(z) \, H(w) (1-\sqrt{q}w)^n B(w)}{z^{k}w^{\ell}}\frac{(z-w)(z-a)(w-a)}{(z^2-1)(w^2-1)(zw-1)}
\end{equation}
where the integrals can be taken over circles with radii larger than $1$ and smaller than $\min\{1/b,1/\sqrt{q}\}$. The change of variable $w\to w^{-1}$ leads to
\begin{equation}
K^{\rm geo}_{11}(k,\ell) =\frac{1}{(2\pi\I)^2}\oint dw \oint dz \frac{w^{\ell-1}}{z^{k}}\frac{h^{\rm geo}_{11}(z,w)}{z-w}
\end{equation}
where we can take $|w|\in (\max\{\sqrt{q},b\},1)$ and $|z|\in (1,\min\{1/b,1/\sqrt{q}\})$.

The integrand behaves like $w^{N-n+\ell-1}$ as $w\to 0$ and thus there is no pole at $w=0$ for any $\ell\geq 1$ and $n\leq N$. Also, it behaves like $1/(z^{N-n+k+1})\lesssim z^{-2}$ as $z\to\infty$ for all $n\leq N$ and $k\geq 1$. Thus there is no pole in $z$ at $\infty$. This means that we can deform the integration contours to include only the poles at $b$ and $\sqrt{q}$. Moreover, we can open up the $z$ contour and close it down around the poles at $1/b$ and $1/\sqrt{q}$. By doing this we change of orientation of the contour, which then gives a minus sign as all the closed contours we consider as anticlockwise oriented. Thus we get
\begin{equation}\label{eqC13}
K^{\rm geo}_{11}(k,\ell) =\frac{-1}{(2\pi\I)^2}\oint\limits_{\Gamma_{b,\sqrt{q}}} dw\!\!\!\!\! \oint\limits_{\Gamma_{1/b,1/\sqrt{q}}}\!\!\!\!\! dz \frac{w^{\ell-1}}{z^{k}}\frac{h^{\rm geo}_{11}(z,w)}{z-w}
\end{equation}
The claimed result follows from the observation that
\begin{equation}
\oint\limits_{\Gamma_b} dw \oint\limits_{\Gamma_{1/b}}dz  \frac{w^{\ell-1}}{z^{k}}\frac{h^{\rm geo}_{11}(z,w)}{z-w} =0.
\end{equation}
\end{proof}

\begin{lem}
  Let $a,b\sqrt{q}$ be all different and $k,\ell\geq 1$. The kernel entry $K^{\rm geo}_{12}$ can be expressed as
\begin{equation}
K^{\rm geo}_{12}(k,\ell) =\frac{-1}{(2\pi\I)^2}\oint\limits_{\Gamma_{\sqrt{q},a,b}}\!\! dw\oint\limits_{\Gamma_{1/\sqrt{q}}} dz \frac{w^{\ell-1}}{z^{k}}\frac{h^{\rm geo}_{12}(z,w)}{z-w} +
\frac{-1}{(2\pi\I)^2}\oint\limits_{\Gamma_{\sqrt{q},a}}  dw \oint\limits_{\Gamma_{1/b}}  dz \frac{w^{\ell-1}}{z^{k}}\frac{h^{\rm geo}_{12}(z,w)}{z-w}
\end{equation}
where
\begin{equation}
h^{\rm geo}_{12}(z,w)=\frac{H(z)B(z)}{H(w)B(w)}\frac{(1-\sqrt{q} z)^n}{(1-\sqrt{q}/w)^n} \frac{(zw-1)(z-a)}{(z^2-1)(w-a)}.
\end{equation}
\end{lem}

\begin{proof}
This time the change of variables is not necessary. Initially we have integration paths such that $\max\{\sqrt{q},a,b\}<|w|<|z|$ and $1<|z|<\min\{1/b,1/\sqrt{q}\}$. As before, $0$ is not a pole for $w$ and $\infty$ is not a pole for $z$. Thus we can deform the $w$ and $z$ contours to include only $\sqrt{q},a,b$ and $1/b,1/\sqrt{q}$ respectively. Also in this case the contribution of the poles at $(w=b,z=1/b)$ is equal to zero. We then get the claimed result.
\end{proof}

\begin{lem}
  Let $a,b,\sqrt{q}$ be all different and $k,\ell\geq 1$. The kernel entry $K^{\rm geo}_{22}$ can be expressed as
\begin{equation}
\begin{aligned}
K^{\rm geo}_{22}(k,\ell) =&\, E(k,\ell)+\frac{-1}{(2\pi\I)^2}\oint\limits_{\Gamma_{\sqrt{q}}}\ \,dw\!\!\!\!\!\! \oint\limits_{\Gamma_{1/\sqrt{q},1/a,1/b}}\!\!\!\!\!\! dz \frac{w^{\ell-1}}{z^{k}}\frac{h^{\rm geo}_{22}(z,w)}{z-w} \\
&+ \frac{-1}{(2\pi\I)^2}\oint\limits_{\Gamma_{a}}\, dw\!\!\!\!\!\!\oint\limits_{\Gamma_{1/\sqrt{q},1/b}}\!\!\!\!\!\! dz \frac{w^{\ell-1}}{z^{k}}\frac{h^{\rm geo}_{22}(z,w)}{z-w}
+\frac{-1}{(2\pi\I)^2}\oint\limits_{\Gamma_{b}}\, dw\!\!\!\!\!\! \oint\limits_{\Gamma_{1/\sqrt{q},1/a}}\!\!\!\!\!\! dz \frac{w^{\ell-1}}{z^{k}}\frac{h^{\rm geo}_{22}(z,w)}{z-w}
\end{aligned}
\end{equation}
where
\begin{equation}
h^{\rm geo}_{22}(z,w)=\frac{H(z)B(z)}{H(w)B(w)}\frac{1}{(1-\sqrt{q}/z)^n(1-\sqrt{q}/w)^n} \frac{(zw-1)}{(w-a)(1-az)}
\end{equation}
and
\begin{equation}
E(k,\ell)=-\sgn(k-\ell)\!\!\! \oint\limits_{\Gamma_{1/a,1/\sqrt{q}}}\!\!\! \frac{dz}{2\pi\I} z^{\ell-k} \frac{z^{-1}-z}{(1-az)(z-a)}\frac{1}{(1-\sqrt{q}/z)^n (1-\sqrt{q} z)^n}.
\end{equation}
\end{lem}

\begin{rem}
The condition that $a,b,\sqrt{q}$ are all different is not really a restriction. If two or more of them are equal, one just has to use a different contour decomposition starting from the integration contours for $(w,z)$ as $\Gamma_{\sqrt{q},a,b}\times\Gamma_{1/\sqrt{q},1/a,1/b}$. For example, if $a=b$ and recalling that $b\in (0,1)$, the given contours are already fine. The decomposition we mention is useful since we want to study the case $b\to 1/a$ later.
\end{rem}

\begin{proof}
We start with
\begin{equation}
K^{\rm geo}_{22}(k,\ell)=\oint \frac{dw}{2\pi\I}\oint \frac{dz}{2\pi\I} \frac{z^{k-1} w^{\ell-1}}{H(z)B(z)H(w)B(w)}\frac{1}{(1-\sqrt{q}z)^n(1-\sqrt{q}w)^n}\frac{z-w}{(zw-1)(z-a)(w-a)}
\end{equation}
where the integration contours can be taken as circles with $|z|=|w|>\max\{\sqrt{q},a,b,1\}$ (the value $1/\sqrt{q}$ is not a pole). The change of variable $z\to z^{-1}$ leads to
\begin{equation}
K^{\rm geo}_{22}(k,\ell)=\frac{1}{(2\pi\I)^2} \oint dw \oint dz \frac{w^{\ell-1}}{z^{k}} \frac{h^{\rm geo}_{22}(z,w)}{z-w}
\end{equation}
where now the integration contours need to satisfy $\max\{\sqrt{q},a,b\}<|z|<|w|<\min\{1/b,1/a,1/\sqrt{q}\}$. Notice however that this time we cannot directly shrink the $w$ contour around $\sqrt{q},a,b$ since $z$ is also a pole for $w$ and lies inside the $w$ contour. We can exchange the contours, that is get $w$ inside $z$, by correcting with the residue at $w=z$. We get
\begin{equation}
K^{\rm geo}_{22}(k,\ell)=E(k,\ell)+\frac{1}{(2\pi\I)^2} \oint dw \oint dz \frac{w^{\ell-1}}{z^{k}} \frac{h^{\rm geo}_{22}(z,w)}{z-w}
\end{equation}
where the integration contours now satisfy $\max\{\sqrt{q},a,b\}<|w|<|z|<\min\{1/b,1/a,1/\sqrt{q}\}$, and where
\begin{equation}
E(k,\ell) = \!\!\oint\limits_{\Gamma_{0, a, \sqrt{q}}}\!\! \frac{dz}{2\pi\I} \frac{z^{\ell-k-1}(1-z^2)}{(1-\sqrt{q}/z)^n(1-\sqrt{q} z)^n}\frac{1}{(1-az)(z-a)}.
\end{equation}

For the double integral, we verify that $w=0$ and $z=\infty$ are not poles and then deform them to include only the poles at $\sqrt{q},a,b$ for $w$ and $1/\sqrt{q},1/a,1/b$ for $z$ (the minus sign comes from the exchange of the orientation in the $z$ path). Then we notice that the contributions of the poles at $(w=a,z=1/a)$ and at $(w=b,z=1/b)$ are $0$. We are left with $9-2=7$ contributions which we put together as three integrals in such a way that the paths can be taken as single contours and do not cross (recall the conditions $ab<1$, $a\sqrt{q}<1$ and that $a,b,\sqrt{q}$ are all different).

Concerning $E(k,\ell)$, it is an easy computation to verify that $E(k,k)=0$ (the $n=0$ and $n \geq 1$ are two separate computations), which is actually a consequence of the anti-symmetry of the entry $K_{22}(k, \ell)$. Furthermore, for $k>\ell$ we have
\begin{equation}
E(k,\ell)=-\!\!\!\!\!\!\oint\limits_{\Gamma_{1/a,1/\sqrt{q}}}\!\!\!\!\!\! \frac{dz}{2\pi\I} z^{\ell-k} \frac{z^{-1}-z}{(1-az)(z-a)}\frac{1}{(1-\sqrt{q}/z)^n (1-\sqrt{q} z)^n}
\end{equation}
and for $k<\ell$ we have ($0$ being no longer a pole)
\begin{equation}
E(k,\ell)=\!\!\oint\limits_{\Gamma_{a,\sqrt{q}}}\!\! \frac{dz}{2\pi\I} z^{\ell-k} \frac{z^{-1}-z}{(1-az)(z-a)}\frac{1}{(1-\sqrt{q}/z)^n (1-\sqrt{q} z)^n} = -E(\ell,k).
\end{equation}
\end{proof}

\subsection{From the geometric to the exponential model: proof of Theorem~\ref{thm:exp_corr}}

The exponential model is obtained from the geometric one by taking the $\epsilon \to 0$ limit with the following variables:
\begin{equation}
(a,b)=(1-\epsilon \alpha,1-\epsilon \beta),\quad q=1-\epsilon ,\quad (k,\ell)=\epsilon ^{-1}(x,y).
\end{equation}
Let us consider the accordingly rescaled kernel and conjugated kernel
\begin{alignat}{2}
K^\epsilon _{11}(x,y)& =\epsilon^{-2-2n}K^{\rm geo}_{11}(k,\ell), &\quad K^\epsilon_{12}(x,y) &= \epsilon^{-1}K^{\rm geo}_{12}(k,\ell), \nonumber \\
K^\epsilon_{21}(x,y)& =\epsilon^{-1}K^{\rm geo}_{21}(k,\ell), &\quad K^\epsilon_{22}(x,y) &= \epsilon^{2n}K^{\rm geo}_{22}(k,\ell).
\end{alignat}
Then the following result is straightforward and does not require complicated asymptotic analysis.

\begin{prop}\label{prop:ConvGeom}
Uniformly for $x,y$ in a bounded set of $\R_+$, we have
\begin{equation}
\lim_{\epsilon\to 0} K^\epsilon_{ij}(x,y)= K_{ij}(x,y)
\end{equation}
where the kernel $K$ is the one of Theorem~\ref{thm:exp_corr} for the exponential model.
\end{prop}

\begin{proof}
Let us also change the integration variables as $z=1+Z \epsilon$, $w=1+W \epsilon$. Notice that all the variables are in an $\epsilon$-neighborhood of $1$ and thus the contours for $Z$ and $W$ can be fixed independently of $\epsilon$; they just need to pass/be in the correct position with respect to the poles.

Therefore we can easily get the following point-wise limits:
\begin{equation}
\begin{aligned}
\lim_{\epsilon\to 0} H(z) z^{-k}& =\left[ \frac{\tfrac12+z}{\tfrac12-z} \right]^{N-1} e^{-x Z}=\Phi(x,Z),\\
\lim_{\epsilon\to 0} (1-\sqrt{q} z)^n \epsilon^{-n} &= (\tfrac12-Z)^n,\\
\lim_{\epsilon\to 0} (1-\sqrt{q} /w)^n \epsilon^{-n} &= (\tfrac12+W)^n
\end{aligned}
\end{equation}
and
\begin{equation}
\begin{aligned}
\lim_{\epsilon\to 0} \frac{1-b/z}{1-b z}\frac{1-b w}{1-b/w}\frac{(z-a)(1-w a)(zw-1)}{(z^2-1)(1-w^2)(z-w)}&=\frac{(Z+\alpha)(Z+\beta)(W-\alpha)(W-\beta)(Z+W)}{4Z W (Z-W)(W+\beta)(Z-\beta)},\\
\lim_{\epsilon\to 0} \frac{1-b/z}{1-b z}\frac{1-b w}{1-b/w}\frac{(z-a)(zw-1)}{(z^2-1)(w-a)(z-w)} \epsilon &=\frac{(Z+\alpha)(Z+\beta)(W-\beta)(Z+W)}{2Z (Z-W)(W+\alpha)(W+\beta)(Z-\beta)},\\
\lim_{\epsilon\to 0} \frac{1-b/z}{1-b z}\frac{1-b w}{1-b/w}\frac{zw-1}{(1-a z)(w-a)(z-w)} \epsilon^{2} &=\frac{-(Z+\beta)(W-\beta)(Z+W)}{(Z-W)(W+\alpha)(W+\beta)(Z-\alpha)(Z-\beta)},\\
\lim_{\epsilon\to 0} z^{\ell-k} \frac{z^{-1}-z}{(1-a z)(z-a)} \epsilon &=e^{-Z(x-y)}\frac{2Z}{Z^2-\alpha^2}.
\end{aligned}
\end{equation}
By choosing the paths for $Z,W$ so that their distances to any of the poles are bounded away from $0$ uniformly for all $0<\epsilon\ll 1$, we have that the integrands (appropriately multiplied by some powers of $\epsilon$) are uniformly bounded. Thus we can use dominated convergence to take the $\epsilon\to 0$ limit inside the integration.
\end{proof}

It remains for us to find appropriate bounds in order to prove the convergence of the Fredholm Pfaffian to its corresponding limit. We do this next.

\begin{lem}\label{lem:Bounds}
Let $\beta\in (0,1/2)$, $\alpha\in (-1/2,1/2)$ with $\alpha+\beta>0$. Let us set
\begin{equation}
\delta=(\beta-\max\{0,-\alpha\})/4,\textrm{ and }\mu =\max\{0,-\alpha\}+2\delta.
\end{equation}
Then the following bounds hold uniformly in $x,y\in\R_+$,
\begin{alignat}{2}
|K^\epsilon_{11}(x,y) e^{\mu (x+y)}| &\leq C e^{-\delta (x+y)}, &\quad |K^\epsilon_{12}(x,y) e^{\mu x} e^{-\mu y}| &\leq C e^{-\delta (x+y)}, \nonumber \\
|K^\epsilon_{21}(x,y) e^{-\mu x} e^{\mu y}| &\leq C e^{-\delta (x+y)}, &\quad |K^\epsilon_{22}(x,y)e^{-\mu (x+y)}| &\leq C e^{-\delta (x+y)}.
\end{alignat}
\end{lem}

\begin{proof}
The proof is quite simple, since the $k,\ell$ dependence is only through the term $w^\ell/z^k$ and the rest of the integrand remains nicely bounded and converges as we saw in Proposition~\ref{prop:ConvGeom}.

For $K^\epsilon_{11}(k, \ell)$, choose contours satisfying $|z|\geq 1+(\mu+\delta)\epsilon$ and $|w|\leq 1-(\mu+\delta)\epsilon$. This gives $|w^\ell/z^k|\leq \frac{(1-(\mu+\delta)\epsilon)^{y/\epsilon}}{(1+(\mu+\delta)\epsilon)^{x/\epsilon}}\simeq e^{-\mu(x+y)} e^{-\delta(x+y)}$.

For $K^\epsilon_{12}(k, \ell)$, choose the contours such that $|z|\geq 1+(\mu+\delta)\epsilon$ and $|w|\leq 1+(\mu-\delta)\epsilon$. This gives $|w^\ell/z^k|\lesssim e^{-\mu(x-y)} e^{-\delta(x+y)}$.

For the double integrand in $K^\epsilon_{22}(k, \ell)$ there are several terms. In the first and third one, we choose contours satisfying $|z|\geq 1-(\mu-\delta)\epsilon$ and $|w|\leq 1-(\mu+\delta)\epsilon$, while in the second one $|z|\geq 1+(\mu+\delta) \epsilon$ and $|w|\leq 1+(\mu-\delta)\epsilon$. Combining the cases, we have $|w^\ell/z^k|\lesssim e^{\mu(x+y)} e^{-\delta(x+y)}$. Finally, for the term coming from $E(k,\ell)$, for $k>\ell$ we take a contour with $|z|\geq 1-(\mu-\delta)\epsilon$, which gives $|z^{\ell-k}|\lesssim e^{(\mu-\delta)(x-y)}\leq e^{\mu x} e^{-\delta(x+y)}$ by using $\mu-\delta\geq \delta$. For $k<\ell$ the bound follows from the anti-symmetry of $E(k,\ell)$.
\end{proof}

\begin{proof}[Proof of Theorem~\ref{thm:exp_corr}]
With Proposition~\ref{prop:ConvGeom} and Lemma~\ref{lem:Bounds} in hand, the proof of Theorem~\ref{thm:exp_corr} is standard. One writes the Fredholm Pfaffian as a series expansion and then uses the Hadamard bound. Exactly the same steps are made in the proof of the convergence of the large time limit of Proposition~\ref{prop:cvgOfFredPf} and thus we omit them here for brevity.
\end{proof}

%\bibliographystyle{patplain}
%\bibliography{Biblio}

\end{document}